\documentclass[twoside,intlimits]{amsart}

\usepackage{amssymb}
\usepackage{amsmath}
\usepackage{amssymb}
\usepackage{amsthm}
\usepackage{bbm}
\usepackage{color}

\usepackage[comment,enumitem]{paper_diening}

\numberwithin{equation}{section} 

\newtheorem{theorem}{Theorem}[section]
\newtheorem{lemma}[theorem]{Lemma}

\newtheorem{corollary}[theorem]{Corollary}

\newtheorem{assumption}[theorem]{Assumption}

\newtheorem{remark}[theorem]{Remark}

\numberwithin{equation}{section}

\providecommand{\Div}{\divergence}
\providecommand{\supp}{\support}
\providecommand{skp}[2]{{\langle{#1},{#2}\rangle}}

\providecommand{\wal}{\ensuremath{w^\alpha_\lambda}}

\providecommand{\dx}{\,\mathrm{d}x}

\providecommand{\dz}{\,\mathrm{d}z}
\providecommand{\ds}{\,\mathrm{d}s}
\providecommand{\dt}{\,\mathrm{d}t}

\providecommand{\Oal}{\ensuremath{\mathcal{O}^\alpha_\lambda}}

\providecommand{\Ma}{\ensuremath{\mathcal{M}^\alpha}}
\providecommand{\Mst}{\ensuremath{\mathcal{M}^*}}
\providecommand{\Ms}{\ensuremath{\mathcal{M}^{\sharp,1}}}
\providecommand{\Mas}{\ensuremath{\mathcal{M}^{\alpha,\sharp,1}}}
\providecommand{\Na}{\ensuremath{\mathcal{N}^{\alpha}}}

\newcommand{\seb}[1]{{#1}}

\begin{document}


\title{Parabolic Lipschitz truncation and Caloric Approximation}

\author{L.~Diening, S.~Schwarzacher, B.~Stroffolini, A.~Verde  }

\address{L.~Diening, Institut f\"ur Mathematik,
Albrechtstr. 28a,
49076 Osnabr\"uck, (Germany)}
\email{lars.diening@uni-osnabrueck.de}

\address{ S.~Schwarzacher, Mathematical Institute,  Charles University in Prague, Sokolovska 83, Praha8-Karlin, 186 75,  (Czech Republic)
 }
\email{schwarz@karlin.mff.cuni.cz}

\address{ B.~Stroffolini, Dipartimento di Matematica, Universit{\`a} di
  Napoli,
  Federico II, Via Cintia, I-80126 Napoli (Italy)}
\email{bstroffo@unina.it}

\address{  A.~Verde, Dipartimento di Matematica, Universit{\`a} di
  Napoli, Federico II, Via Cintia, I-80126 Napoli (Italy)}
\email{anverde@unina.it}

\begin{abstract}
  We develop an improved version of the parabolic Lipschitz
  truncation, which allows qualitative control of the distributional
  time derivative and the preservation of zero boundary values.  As a
  consequence, we establish a new caloric approximation lemma. We show
  that almost $p$-caloric functions are close to $p$-caloric
  functions. The distance is measured in terms of spatial gradients as
  well as almost uniformly in time.  Both results are extended to the
  setting of Orlicz growth.
\end{abstract}


\maketitle



\section{Introduction}
\label{sec:introduction}

\noindent
The purpose of the Lipschitz truncation is to regularize a given
function by a Lipschitz continuous one by changing it only on a small
\emph{bad} set. It is crucial for the applications that the function
is not changed globally, which rules out the possibility of
convolutions.
The Lipschitz truncation technique was introduced by
Acerbi-Fusco~\cite{AceF88} to show lower semi-continuity of certain
variational integrals.

Since then this technique has been successfully applied in many
different areas. Let us provide a few examples.  The Lipschitz
truncation was used in the context of biting lemmas, existence theory
and regularity results of non-linear elliptic PDE for example
in~\cite{AceF84} \cite{Zha90}, \cite{BalZha90}, \cite{DuzM04},
\cite{DieStrVer12} and \cite{DieLenStrVer12}. 

It was also successfully applied in the framework of non-Newtonian
fluids of power law type~\cite{FreMS03,DieMS08} and even in the
context of numerical analysis~\cite{DieKreSul12}.  In
\cite{BulDieSch16,BulSch16} the Lipschitz truncation was used to
develop an existence theory of vector valued very weak solutions of
elliptic PDEs.

All of these application have in common that the desired test
functions are a~priori not admissible, but have to be approximated by
Lipschitz functions. In order to preserve things like pointwise
monotonicity of the system, it is important that the truncation takes
place only on the small bad set. The bad set is usually defined in
terms of the level sets of the maximal operator of the gradients.

During these years the Lipschitz truncation technique has been refined
with respect to several aspects. In the stationary situation the
picture is almost complete. It is now possible to preserve zero
boundary value, obtain stability in all~$L^p$-spaces and to apply the
technique to sequences of functions. Moreover, the Lipschitz
truncation can be interpreted as a \Calderon{}-Zygmund
decomposition in the Sobolev spaces of first order, see~\cite{Aus04}.

In the parabolic context the theory is much less developed. The
parabolic Lipschitz truncation was introduced by Kinnunen-Lewis
\cite{KinLew00}. They used it to prove higher integrability for very
weak solutions of the evolutive $p$-Laplacian systems.  On the other
hand, Diening-Ruzicka-Wolf~\cite{DieRuzWol10} developed a parabolic Lipschitz
truncation to show existence of fluids of power law type; i.e. the
evolutive analogue to \cite{FreMS03}. In
\cite{BreDieFuc12,BreDieSch13} a parabolic Lipschitz truncation was
developed, which preserves the solenoidal structure of the given
function and makes the truncation more suitable for problems from
fluids dynamics.

The difficulty of the parabolic Lipschitz truncation in contrast to
the stationary case is due to the fact, that the time-derivative of
the solution is only defined in terms of negative Sobolev spaces or in
the distributional sense. Therefore, the parabolic Lipschitz
truncations mentioned above lacked the possibility to preserve zero
boundary values and to obtain control on the time derivative of the
truncation. In this paper we will overcome both of these problems.

In what follows we will introduce our parabolic Lipschitz truncation in the setting of $p$-growth assumptions. 
The full
statement that holds for general Orlicz growth assumptions can be found in Theorem~\ref{thm:orlicz} in the next section.

Our standing assumption for the Lipschitz truncation, is that the given function $w$ has a time derivative in the following sense:
\begin{align}\label{weak}
  \begin{aligned}
    \partial_t w &= \Div \,G \qquad \text{in $\mathcal{D}'(J
      \times \Omega$)}
  \end{aligned}
\end{align}
where $J$ is a time interval and $\Omega$ is a bounded domain in
$\setR^m$, $m \ge 2$.  We take as \lq\lq bad set" a superlevel set of
the maximal function of the spatial gradient and of the time
derivative in the following way. Let
\begin{align*} 
  \mathcal{O}_{\lambda}^{\alpha}& := { \{\Ma(\chi_{J\times\Omega}\nabla
    w)>\lambda\} \cup \{\alpha \Ma(\chi_{J\times\Omega}G) >
    \lambda\}},
\end{align*}
where $\lambda>0$ and the {$\alpha$-parabolic maximal function
  $\Ma$} is defined using the (backwards in time) parabolic
cylinders 
$Q^{\alpha}_r :=(-\alpha r^2, 0)\times B_r$
in the following way:
\begin{align}\label{maximal}
  (\Ma g)(x) &:= \sup_{Q \in \mathcal{Q}^{\alpha} \,:\, x \in Q}
  \dashint_Q|g|.
\end{align}
where $ \mathcal{Q}^{\alpha}$ is the family of cylinders $ Q^{\alpha}_r, r>0.$

Here $\alpha$ is a scaling quantity, to allow different integrability
assumptions on $\nabla w $ and $G$.  Having collected the necessary
notation we may state the theorem.
\begin{theorem}
  \label{thm:pLip}
  Let $G\in L^{p'}(J\times \Omega)$ and $w\in L^p(J,W^{1,p}_0(\Omega))$ 
  satisfy \eqref{weak}.  
  Then there exists an approximation $\wal\in
  L^{p}(J,W^{1,p}_0(\Omega))$ with the following properties:
  \begin{enumerate}
  \item $\wal=w$ on $({\Oal})^c$.
  \item $\Ma(\nabla \wal)\leq c\,\lambda$, i.e. $\wal$ is
    Lipschitz continuous with respect to space.
 \item 
   \begin{align*}
      \int_{J\times\Omega}\abs{\nabla(\wal-w)}^p\dz\leq
      c\int_{\Oal}\abs{\nabla w}^p+\lambda^p|\Oal|.
    \end{align*}
  \item\label {itm:time1}
 $
     \alpha  \mathcal{ N}^{\alpha}(\partial_t\wal)\le c \lambda
$
where $ \mathcal{ N}^{\alpha}$ is defined in \eqref{maxalpha}.
  \item $\wal$ is Lipschitz continuous with respect to  the scaled,
    parabolic metric, i.e.
    \begin{align*}
      \abs{\wal(t,x) - \wal(s,y)} \leq c\, \lambda \max
      \biggset{ \frac{\abs{t-s}^\frac 12}{\alpha^{\frac 12}}, \abs{x-y} }
    \end{align*}
    for all $(t,x),(s,y)\in J\times \Omega$.
  \item\label{itm:integrationbyparts1} for all $\eta\in C^{\infty}_0(Q) $ it holds:
    \begin{align*}
      \skp{\partial_t w}{\wal \eta}=\frac12 \int_Q (|\wal |^2
      -2w\cdot \wal) \partial_t \eta dz+\int_{\Oal} (\partial_t \wal)
      (\wal-w)\eta dz.
    \end{align*}
    \end{enumerate}
\end{theorem}
Observe, that \ref{itm:time1} shows that our approximation does also
approximate the distributional time- derivative. The maximal operator
$\mathcal{N}^{\alpha}$ is defined in terms of the distributional time
derivative. It seems to be a novel tool to quantify the distributional
time derivative in such a way. In a way the
boundedness of~$\mathcal{N}^\alpha(\partial_t w^\alpha_\lambda)$
corresponds to $\partial_t w^\alpha_\lambda \in L^\infty(J,
W^{-1,\infty}(\Omega))$.

As an application of our parabolic Lipschitz truncation, we present a
new \emph{caloric approximation lemma}.  We show that every
{\lq\lq almost $p$-caloric\rq\rq} function has a $p$-caloric
approximation {\lq\lq close enough\rq\rq}. The following theorem is
the $p$-version of the more general result for Orlicz function, see
Theorem~\ref{thm:almost}.
\begin{theorem}\label{p-caloric}
  Let $p\in (1,\infty)$ and $Q$ be a times-space cylinder, $Q=I\times B=(t^-,t^+)\times B$. Let $\sigma\in (0,1)$,
  $q\in[1,\infty)$ and $\theta\in (0,1)$.  Moreover, let $\tilde{Q}$ be such that $Q \subset
  \tilde{Q} \subset 2Q$.  Then, for all{ $\epsilon>0$} there exists a{
    $\delta>0$} s.t. the following holds: if { $u\in
    L^p(I,W^{1,p}_0(B)), u_t=\Div\, G,\, G\in L^{p'}(J\times
    \Omega)$,} is almost { $p$}-caloric in the sense that for all { $
    \xi\in C_0^\infty(Q),$} 
  \begin{align*} 
    \Big|\dashint_Q \!\!\!\!\!-u \partial_t \xi +|\nabla
    u|^{p-2}\nabla u\nabla \xi dz\Big| \leq \delta \bigg( \dashint_{\tilde{Q}}
    |\nabla u|^p +\abs{G}^{p'}dz+ \|\nabla \xi\|^p_\infty \bigg)
  \end{align*}
  then there exists a {$p$}-caloric function {$h$} s.t. {$h=u$} on
  {$\partial_p Q$} and
  \begin{align*}
    \bigg(\dashint_I\Big(\dashint_B \Big({|u-h|^{2}\over
      |t^+-t^-|}\Big)^{\sigma}dx\Big)^{\frac{q}{
        \sigma}}\dt\bigg)^\frac{1}{q}&+\bigg(\dashint_Q |V(\nabla u )-
    V(\nabla h)|^{2\theta} dz \bigg)^{\frac{1}{\theta}}
    \\
    &\leq \epsilon \dashint_{\tilde{Q}} |\nabla u|^p+\abs{G}^{p'} \,
    dz .
  \end{align*}
where $V(z)=|z|^{\frac{p-2}{2}} z$.
\end{theorem}
\noindent
If $u$ would be $p$-caloric, then we could choose~$\delta=0$ in the
assumption of Theorem~\ref{p-caloric} and $h=u$ as an approximation. The
small parameter~$\delta>0$ indicates, that~$u$ behaves like a small
perturbation of a $p$-caloric function. This smallness however is only
needed in reaction to very regular test functions~$\xi$. Nevertheless,
Theorem~\ref{p-caloric} ensures that~$u$ is close to a $p$-caloric
function~$h$. The closeness is expressed up to a small loss in
the exponent in the natural distance of the $p$-heat equation, which are
$L^\infty(L^2)$ and $L^p(W^{1,p})$.  In particular, we have control
on the distance in the sense of space and time derivatives.

In the stationary case, the method is called \emph{harmonic
  approximation lemma} and its idea goes back to De Giorgi. He used it
in geometric measure theory to prove regularity of harmonic
maps. See~\cite{DuzMin09h} for an overview on the harmonic
approximation lemma. The closeness in the sense of gradients and the
preservation of the boundary values was introduced
in~\cite{DieStrVer12}.

The p-caloric approximation method was developed by B\"ogelein, Duzaar
and Mingione~\cite{BoeDuzMin13}, (see
also~\cite{DuzMin05},\cite{DuzMinStef11}).  We wish to quickly point
the improvements of the approximation lemma here with respect to the one
in~\cite{BoeDuzMin13}. First, our assumptions are weaker: we only
assume \eqref{weak} and we deduce the validity of a Poincar\'e
inequality. Second, our proof is directly and completely avoids any
argument by contradiction. This direct approach via the parabolic
Lipschitz truncations gives us a much finer control on the quantities
and allows us to show closeness both in $L^q(L^{2\sigma})$ and
$L^{p\theta}(W^{1,{p\theta}})$ norms, (the last closeness is via the
natural quantity $V(z)=|z|^{\frac{p-2}{2}} z$).  In addition the
previous estimate measures the closeness of the time derivatives and
spatial gradients in a quantitative way. Third, we can preserve
boundary values, which is very handy for applications.

As mentioned above, the direct proof of harmonic and caloric
approximation lemmas by means of the Lipschitz truncation has many
advantages. Recently, the solenoidal parabolic Lipschitz truncation
of~\cite{BreDieSch13} was used in~\cite{Bre16} to derive an caloric
approximation lemma for the linear, parabolic $\mathcal{A}$-Stokes
problem, which is useful in fluid mechanics. In contrast
to~\cite{Bre16} we can preserve boundary values and treat a non-linear 
equation.

\subsection{Acknowledgments}

These results were announced for the first time at the Mittag-Leffler Institute for the special program \lq\lq Evolutionary problems\rq\rq in 2013.
We would like to thank the institute for the hospitality. S. Schwarzacher wishes to thank program PRVOUK~P47, financed by Charles University in Prague.
B. Stroffolini and A. Verde have been partially supported by the 
Italian M.I.U.R. Project  \lq\lq Calcolo delle Variazioni " (2012).

\section{Parabolic Lipschitz truncation}
\label{sec:truncation}

In this section with derive an improved version of the parabolic
Lipschitz truncation. Earlier versions are due to~\cite{KinLew02} and
\cite{DieRuzWol10}.

We start by assuming that~$w \in L^1(J, W^{1,1}_0(\Omega))$ is a
distributional solution (possible vectorial) to
\begin{align}
  \label{eq:w}
  \begin{alignedat}{2}
    \partial_t w &= \Div G &\qquad &\text{in $\mathcal{D}'(J \times
      \Omega$)}
    \\
    w &= 0 &\qquad &\text{on $\partial_{\text{par}} (J \times \Omega)$}
  \end{alignedat}.
\end{align}
Here $J = (-t_0,0)$ denotes the time interval.  The space
domain~$\Omega \subset \setR^m$ should have the fat complement
property, see Remark~\ref{rem:fatcomplement}. In particular, it
suffices that~$\Omega$ is a bounded open domain with Lipschitz
boundary. In many applications it is enough to consider the case
where~$\Omega$ is a ball or a cube. By $\partial_{\text{par}} (J
\times \Omega)$ we denote the parabolic boundary of~$J\times \Omega =
(\set{-t_0} \times \Omega) \cup (J \times \partial \Omega)$. The
function~$G$ will at least be in~$L^1(J \times \Omega)$. Note that the
zero boundary values on the parabolic boundary are well defined due to
$w \in L^1(J, W^{1,1}_0(\Omega))$ and $\partial_t w \in L^1(J,
(W^{1,\infty}_0(\Omega))^*)$.

\begin{remark}
  \label{rem:fatcomplement}
  It is sufficient for us to consider domains~$\Omega$ that have the
  {\em fat complement property}, i.e. there exists~$A_1 \geq 1$ such
  that for all $x \in \Omega$
  \begin{align}
    \label{eq:ass_bnd}
    \abs{B_{2\,\distance(x, \Omega^\complement)}(x)} &\leq A_1\,
    \abs{B_{2\,\distance(x, \Omega^\complement)}(x) \cap
      \Omega^\complement}.
  \end{align}
  If $\Omega \subset \setR^d$ is an open bounded set with
  Lipschitz boundary then $\Omega$ has the fat complement property.
\end{remark}

Let us recall some 
 definitions and results that are standard in the context of
N-functions.
A real function $\phi \,:\, \setR^{\geq 0} \to
\setR^{\geq 0}$ is said to be an N-function if it satisfies the
following conditions: $\phi(0)=0$ and there exists the derivative
$\phi'$ of $\phi$.  This derivative is right continuous,
non-decreasing and satisfies $\phi'(0) = 0$, $\phi'(t)>0$ for $t>0$,
and $\lim_{t\to \infty} \phi'(t)=\infty$. Moreover, $\phi$ is
convex.

We say that $\phi$ satisfies the $\Delta_2$-condition, if there
exists $c > 0$ such that for all $t \geq 0$ holds $\phi(2t) \leq c\,
\phi(t)$. We denote the smallest possible constant by
$\Delta_2(\phi)$. Since $\phi(t) \leq \phi(2t)$ the $\Delta_2$
condition is equivalent to $\phi(2t) \sim \phi(t)$.

By $L^\phi$ and $W^{1,\phi}$ we denote the classical Orlicz and
Sobolev-Orlicz spaces, i.\,e.\ $f \in L^\phi$ iff $\int
\phi(\abs{f})\,dx < \infty$ and $f \in W^{1,\phi}$ iff $f, \nabla f
\in L^\phi$. By $W^{1,\phi}_0(\Omega)$ we denote the closure of
$C^\infty_0(\Omega)$ in $W^{1,\phi}(\Omega)$.

By $(\phi')^{-1} \,:\, \setR^{\geq 0} \to \setR^{\geq
  0}$ we denote the function
\begin{align*}
  (\phi')^{-1}(t) &:= \sup \set{ s \in \setR^{\geq 0}\,:\,
    \phi'(s) \leq t}.
\end{align*}
If $\phi'$ is strictly increasing then $(\phi')^{-1}$ is the inverse
function of $\phi'$.  Then $\phi^\ast \,:\, \setR^{\geq 0} \to
\setR^{\geq 0}$ with
\begin{align*}
  \phi^\ast(t) &:= \int_0^t (\phi')^{-1}(s)\,ds
\end{align*}
is again an N-function and $(\phi^\ast)'(t) =
(\phi')^{-1}(t)$ for $t>0$. It is the complementary function of
$\phi$.  Note that $\phi^*(t)= \sup_{s \geq 0} (st - \phi(s))$ and
$(\phi^\ast)^\ast = \phi$. For all $\delta>0$ there exists
$c_\delta$ (only depending on $\Delta_2({\phi, \phi^\ast})$
such that for all $t, s \geq 0$ holds
\begin{align}  \label{eq:young}
  t\,s &\leq \delta\, \phi(t) + c_\delta\, \phi^\ast(s),
\end{align}

  This inequality is called {\em
  Young's inequality}. For all $t\geq 0$
\begin{gather}
  \label{ineq:phiast_phi_p_pre}
  \begin{aligned}
    \frac{t}{2} \phi'\Big(\frac{t}{2} \Big) \leq \phi(t)
    \leq t\,\phi'(t),
    \\
    \phi \bigg(\frac{\phi^\ast(t)}{t} \bigg) \leq \phi^\ast(t) \leq \phi
    \bigg( \frac{2\, \phi^\ast(t)}{t} \bigg).
  \end{aligned}
\end{gather}
Therefore, uniformly in $t\geq 0$
\begin{gather}
  \label{ineq:phiast_phi_p}
  \phi(t) \sim \phi'(t)\,t, \qquad
  \phi^\ast\big( \phi'(t) \big) \sim \phi(t),
\end{gather}
where the constants only depend on $\Delta_2(\phi, \phi^\ast)$.

We will assume that  $\phi$ satisfies the
following assumption.
\begin{assumption}
  \label{ass:phi}
  Let $\phi$ be an N-function such that
  $\phi$ is $C^1$ on $[0,\infty)$ and $C^2$ on $(0,\infty)$. Further
  assume that
  \begin{align}
    \label{eq:phi_pp}
    \phi'(t) &\sim t\,\phi''(t)
  \end{align}
  uniformly in $t > 0$. The constants in~\eqref{eq:phi_pp} are called
  the {\em characteristics of~$\phi$}.
\end{assumption}
We remark that under these assumptions $\Delta_2({\phi,\phi^\ast})
< \infty$ will be automatically satisfied, where
$\Delta_2({\phi,\phi^*})$ depends only on the characteristics
of~$\phi$. 

For given $\phi$ we define the associated N-function $\psi$ by
\begin{align}
  \label{eq:def_psi}
  \psi'(t) &:= \sqrt{ \phi'(t)\,t\,}.
\end{align}

We remark that if $\phi$ satisfies
Assumption~\ref{ass:phi}, then also $\phi^*$, $\psi$, and $\psi^*$
satisfy this assumption.

The idea of the parabolic Lipschitz truncation is to cut certain
maximal functions of the gradient and the time derivative. Since the
time derivative is only defined in the weak sense by~$\partial_t w =
\divergence G$, we will cut the maximal operator of~$G$
instead of~$\partial_t w$.

The properties of the Lipschitz truncation are
summarized in the following theorem.

\begin{theorem}
  \label{thm:orlicz}
  Let $w\in L^1(J,W^{1,1}_0(\Omega))$ and $\nabla w\in L^\phi(J\times\Omega)$
  satisfies \eqref{eq:w}.
  For~$\lambda,\alpha>0$ define the {\em bad
    set} $\Oal$ by
  \begin{align} 
    \label{eq:Oal}
    \mathcal{O}_{\lambda}^{\alpha}& := { \{\Ma(\chi_{J\times\Omega}\nabla
      w)>\lambda\} \cup \{\alpha \Ma(\chi_{J\times\Omega}G) >
      \lambda\}},
  \end{align}
  
  Then there exists an approximation $\wal\in
  L^{\phi}(J,W^{1,\phi}_0(\Omega))$ with the following properties:
  \begin{enumerate}
  \item $\wal=w$ on $({\Oal})^c$.
  \item $\Ma(\nabla \wal)\leq c\,\lambda$, i.e. $\wal$ is
    Lipschitz with respect to space.
 \item 
   \begin{align*}
      \int_{J\times\Omega}\phi(\abs{\nabla(\wal-w)})\dz\leq
      c\int_{\Oal}\phi(\abs{\nabla w})+\phi(\lambda)|\Oal|.
    \end{align*}
  \item\label {itm:time}
 $
     \alpha  \mathcal{ N}^{\alpha}(\partial_t\wal)\le c \lambda
$
where $ \mathcal{ N}^{\alpha}$ is defined in \eqref{maxalpha}.
  \item $\wal$ is Lipschitz continuous with respect to  the scaled,
    parabolic metric, i.e.
    \begin{align*}
      \abs{\wal(t,x) - \wal(s,y)} \leq c\, \lambda \max
      \biggset{ \frac{\abs{t-s}^\frac 12}{\alpha^{\frac 12}}, \abs{x-y} }
    \end{align*}
    for all $(t,x),(s,y)\in J\times \Omega$.
  \item\label{itm:integrationbyparts} for $J=(t^-,t^+)$ and arbitrary $\eta\in W^{1,\infty}_0(-\infty,t^+)$ it holds:
    \begin{align*}
      \skp{\partial_t w}{\wal \eta}=\frac12 \int_Q (|\wal |^2
      -2w\cdot \wal) \partial_t \eta dz+\int_{\Oal} (\partial_t \wal)
      (\wal-w)\eta dz
    \end{align*}
    \end{enumerate}
\end{theorem}
The proof will be achieved through several lemmas.

\subsection{Parabolic \Poincare{} type inequality}
\label{ssec:poincare}

The goal of this subsection is to derive a very weak form of the
parabolic \Poincare{} inequality on parabolic cylinders, where the
time derivative is just defined in a weak sense, see
Theorem~\ref{thm:poincare_par}. 

We start with some notations.  By $B_r(x)$, resp. $I_r(t)$, we denote
the standard euclidean ball with radius~$r$ and center~$x\in \setR^m$,
resp.~$t \in \setR$.  For $\alpha>0$ define the $\alpha$-parabolic
metric $d_\alpha: \setR \times \setR^m \to [0,\infty)$ by
\begin{align*}
  d_\alpha\big((t,x),(\tau,y)) := \max \bigset{ \alpha^{-\frac 12}
    \abs{t-\tau}^{\frac 12}, \abs{x-y}}.
\end{align*}
The balls with radius~$r$ respect to~$d_\alpha$ are called
$\alpha$-parabolic cylinders with radius~$r$. Any $\alpha$-parabolic
cylinder~$Q$ can be represented in terms of euclidean balls, i.e.
\begin{align*}
  Q = Q_r^{\alpha}(t,x) &:= I_{\alpha r^2}(t) \times B_r(x) = I \times B.
\end{align*}
for some $(t,x) \in \setR^{m+1}$, where $r$ is the radius of~$Q$.


\noindent By $\sigma Q$ (for $\sigma>0$) we denote the parabolic scaled cylinder
with the same center but $\sigma$-times the radius with respect
to~$d_\alpha$. In particular, for $Q=I \times B$ we have $\sigma Q =
(\sigma^2I) \times (\sigma B)$.
We denote by $\abs{E}$  the Lebesque measure of $E$ for a measurable set $E$ and by  $\chi_E$  its characteristic function. 
We define 
\[
\dashint_E\abs{f}\dx=:\frac1{\abs{E}}\int_E \abs{f} \dx.
\]
For a non-negative integrable function $\eta$ we define
\[
 \mean{f}_\eta:=\frac1{\norm{\eta}_{1}}\int f\eta\dx
\]
and for a measurable set $E$ we define
$\mean{f}_E:=\mean{f}_{\chi_E}$. The integration is taken over the
natural domain of~$f$, so if $f$ is defined on~$Q$, then the integral
is over~$Q$.

\noindent We need the following version of the norm conjugate formula for
$L^1_0(I)$.

\begin{lemma}
  \label{lem:ncf}
  Let $f \in L^1(I)$, then
  \begin{align*}
    \int_I \abs{f-\mean{f}_I}\dt &\leq 2\,\sup_{\beta \in
      C^\infty_{0,0}(I), \norm{\beta}_\infty \leq 1} \int_I f \beta
    \dt \leq 2\, \int_I \abs{f-\mean{f}_I}\dt.
  \end{align*}
\end{lemma}

\begin{proof}
  The second estimate is obvious, so we just need to prove the first
  one.  It suffices to prove the case $I= (0,1)$. Fix $\delta>0$. Then
  due to the isometry $(L^1(I))^* = L^\infty(I)$, we can find $g \in
  L^\infty(I)$ with $\norm{g}_\infty \leq 1$, such that
  \begin{align}
    \label{eq:ncf1}
    \int_I \abs{f-\mean{f}_I}\dt &\leq \delta + \int_I (f-\mean{f}_I)
    g\dt = \delta + \int_I f (g- \mean{g}_I)\dt.
  \end{align}
  For $\epsilon\in (0, \frac 14)$ define $I_\epsilon =
  (\epsilon,1-\epsilon)$. Let $\psi_\epsilon$ denote a standard
  mollifier with $\supp \psi_\epsilon \subset B_\epsilon(0)$. Define
  \begin{align*}
    h_\epsilon &:= \big(\chi_{I_\epsilon} (g- \mean{g}_{I_\epsilon})
    \big) *\psi_{\epsilon/2}.
  \end{align*}
  It is easy to see that $h_\epsilon \in C^\infty_{0,0}(I)$,(subspace of $C^{\infty}_0$ whose elements have mean value zero),
  $h_\epsilon \to g - \mean{g}_I$ almost everywhere for $\epsilon\to 0$, $\norm{h_\epsilon}_{L^\infty(I)} \leq 2\,\norm{g}_\infty$. In particular, it follows by the
  dominated convergence theorem that
  \begin{align*}
    \int_I f (g- \mean{g}_I)\dt &= \lim_{\epsilon \to 0} \int_I f
    h_\epsilon\dt.
  \end{align*}
  This and~\eqref{eq:ncf1} imply
  \begin{align*}
    \int_I \abs{f-\mean{f}_I}\dt &\leq \delta + \sup_{h_\epsilon \in
      C^\infty_{0,0}(I), \norm{h_\epsilon}_\infty \leq 2} \int_I f
    h_\epsilon \dt.
  \end{align*}
  The claim follows, since $\delta>0$ was arbitrary.
\end{proof}

\begin{lemma}
  \label{lem:ncf2}
  Let $f \in L^1(I)$, then
  \begin{align*}
    \int_I \abs{f-\mean{f}_I}\dt &\leq 2\,\sup_{\gamma \in
      C^\infty_{0}(I), \norm{\gamma'}_\infty \leq 1} \biggabs{\int_I f
    \gamma' \dt} \leq 2\, \int_I \abs{f-\mean{f}_I}\dt.
  \end{align*}
\end{lemma}
\begin{proof}
  This follows immediately from Lemma~\ref{lem:ncf}. Indeed, if $\beta
  \in C^\infty_{0,0}(I)$, then its primitive $\gamma(t):=
  \int_{-\infty}^t \beta(s)\ds$ satisfies $\gamma \in
  C^\infty_{0,0}(I)$. On the other hand for every $\gamma \in
  C^\infty_0(I)$, we have $\gamma' \in C^\infty_{0,0}(I)$.
\end{proof}

For an $\alpha$-parabolic cylinder~$Q=Q_r =I_{\alpha r^2} \times B_r$
we define
\begin{align*}
  \mathcal{F}_Q := \set{ \xi \in C^\infty_0(Q)\,:\,
    \norm{\xi}_{\mathcal{F}_Q} := \norm{\xi}_\infty + r \norm{\nabla
      \xi}_\infty + \alpha r^2 \norm{\partial_t \xi}_\infty \leq 1},
\end{align*}
Define
\begin{align*}
  \mathcal{M}_Q(a) &:= \dashint_Q \abs{a} \dz,
  \\
  \Ms_Q(a) &:= \dashint_Q \frac{\abs{a - \mean{a}_Q}}{r_Q} \dz
\end{align*}

\noindent For a distribution~$a \in \mathcal{D}'(Q)$ we define
\begin{align*}
  \mathcal{N}_Q(a) &:= \sup_{\xi \in \mathcal{F}_Q} \big( r\, \abs{Q}^{-1}
  \abs{\skp{a}{\xi}} \big).
\end{align*}
We use the letter $\mathcal{N}$ for ``negative'', since we measure
somehow the local information on~$\partial_t a$ in a negative space.
We can observe that 
\begin{equation}
  (\Ma a)(x) = \sup_{Q \in \mathcal{Q}^{\alpha} \,:\, x \in Q} \mathcal{M}_Q(a),
 \end{equation}

We also define the maximal operator
\begin{align}\label{maxalpha}
  (\Na a)(x) &:= \sup_{Q \in \mathcal{Q}^{\alpha} \,:\, x \in Q}
  \mathcal{N}_Q(a). 
\end{align}

\begin{remark}
  \label{rem:NvsMdiv}
  If $\partial_t a = \divergence G$ on~$Q$, then
  \begin{align*}
    \mathcal{N}_Q(\partial_t a) &= \sup_{\xi \in \mathcal{F}_Q} \big(
    r\, \abs{Q}^{-1} \abs{\skp{\partial_t a}{\xi}} \big)
    \\
    &= \sup_{\xi \in \mathcal{F}_Q} \big( r\, \abs{Q}^{-1}
    \abs{\skp{G}{\nabla \xi}} \big)
    \\
    &\leq \dashint_Q \abs{G}\dz.
  \end{align*}
\end{remark}

We need the following version of parabolic \Poincare{}'s inequality with respect to time.
\begin{lemma}
  \label{lem:poincare_time}
  Let $\eta \in C^\infty_0(B)$ with $\eta \geq 0$, $\int_B \eta(x)
  \dx>0$ and $\norm{\eta}_\infty + r\, \norm{\nabla \eta}_\infty \leq
  c_0\, \abs{B}^{-1} \norm{\eta}_1$.  Then for every
  $\alpha$-parabolic cube $Q = I \times B$ we have
  \begin{align*}
    \dashint_{I} \bigabs{\mean{a(t)}_{\eta}- \mean{a}_{\eta\times I}}
    \dt &\leq c\, r\alpha \mathcal{N}_Q(\partial_t a),
  \end{align*}
  where $c$ depends on~$\eta$ only through~$c_0$. Here we use the
  notation $\mean{a}_{\eta\times I}= \frac{1}{|I|}\int_I
  \mean{a(t)}_{\eta} dt$.
\end{lemma}
\begin{proof}
  We can assume without loss of generality that $\int_B \eta(x)\,dx =
  1$.  From Lemma~\ref{lem:ncf2} it follows that
  \begin{align*}
    \dashint_{I} \bigabs{\mean{a(t)}_{\eta}- \mean{a}_{\eta\times I}}
    \dt &\leq 2\, \sup_{\gamma \in C^\infty_{0}(I),
      \norm{\gamma'}_\infty \leq 1} \biggabs{\dashint_I \mean{a(t)}_\eta
      \gamma'(t) \dt}
    \\
    &= 2 \, \abs{B} \sup_{\gamma \in
      C^\infty_{0}(I), \norm{\gamma'}_\infty \leq 1} \biggabs{\dashint_Q
      a\, \partial_t (\eta \gamma) \dz}.
  \end{align*}
  We want to estimate the integral in the last expression by means of
  $\mathcal{N}_Q(\partial_t a)$. Let~$\gamma \in
  C^\infty_0(I)$ with $\norm{\gamma'}_\infty \leq 1$. Then
  $\norm{\gamma}_\infty \leq c \abs{I}$. We estimate
  \begin{alignat*}{3}
    \norm{\eta \gamma}_\infty &\leq \norm{\eta}_\infty
    \norm{\gamma}_\infty &&\leq c_0\, \abs{B}^{-1} \abs{I},
    \\
    r\norm{\nabla(\eta \gamma)}_\infty &\leq r \norm{\nabla
      \eta}_\infty \norm{\gamma}_\infty &&\leq c_0\, \abs{B}^{-1}
    \abs{I},
    \\
    \alpha r^2 \norm{\partial_t(\eta \gamma)}_\infty &\leq
    \norm{\eta}_\infty \alpha r^2 \norm{\partial_t \gamma}_\infty
    &&\leq c_0\, \abs{B}^{-1} \abs{I}.
  \end{alignat*}
  In particular, $\norm{\eta \gamma}_{\mathcal{F}_Q} \leq c_0
  \abs{B}^{-1} \abs{I} = c_0 \abs{B}^{-1} \alpha r^2$.  Therefore,
  using the definition of $\mathcal{N}_Q(\partial_t a)$ we have
  \begin{align*}
    \abs{B} \biggabs{\dashint_Q a\, \partial_t (\eta \gamma) \dz}
    &\leq c\, \abs{B} \, r^{-1} \mathcal{N}_Q(\partial_t a)\,
    \norm{\eta \gamma}_{\mathcal{F}_Q} \leq c\, \alpha r
    \mathcal{N}_Q(\partial_t a).
  \end{align*}
  and the claim follows.
\end{proof}
We are now in a position to state the following Poincar\'e inequality :
\begin{theorem}
  \label{thm:poincare_par}
  Let $Q=I \times B$ be $\alpha$-parabolic cube and let $\rho \in
  L^1(Q)$ be such that $\rho \geq 0$ and $\norm{\rho}_\infty \leq c_0
  \abs{Q}^{-1} \norm{\rho}_1$. Then
  \begin{align*}
    \dashint_Q \biggabs{ \frac{a- \mean{a}_{\rho}}{r} } \dz
    &\leq c\, \dashint_Q \abs{\nabla a} \dz + c\, \alpha\,
    \mathcal{N}_Q(\partial_t a).
  \end{align*}
  Recall that $\mathcal{N}_Q(\partial_t a) \leq \dashint_Q \abs{G}
  \dz$ if $\partial_t u = \divergence G$ with $G\in L^1(Q)$, due to
  Remark~\ref{rem:NvsMdiv}.
\end{theorem}
\begin{proof}
  We begin with the special case~$\rho = \chi_I \eta$ with~$\eta$ as in
  Lemma~\ref{lem:poincare_time}.
  \begin{align*}
    \dashint_{Q} \biggabs{\frac{a-\mean{a}_{\eta\times I}}{r}}\dz
    &\leq \dashint_{I} \dashint_{B}
    \biggabs{\frac{a-\mean{a(t)}_{\eta}}{r}}\dx \dt + \dashint_{I}
    \biggabs{\frac{\mean{a(t)}_{\eta}- \mean{a}_{\eta\times I}
      }{r}}\dt
    \\
    &=: I + II.
  \end{align*}
  Now the claim follows by using \Poincare{} in space for the first
  term and Lemma~\ref{lem:poincare_time} for the second term.

  Now consider the case of arbitrary~$\rho$ as in the
  assumptions. Then
  \begin{align*}
    \dashint_{Q} \bigabs{a-\mean{a}_\rho}\dz &\leq \dashint_{Q}
    \bigabs{a-\mean{a}_{\eta\times I}}\dz +
    \bigabs{\mean{a}_\rho-\mean{a}_{\eta\times I}}.
  \end{align*}
  Now Jensen's inequality with respect to the integration of
  $\mean{a}_\rho$ together with the assumptions on~$\rho$ imply
  \begin{align*}
    \bigabs{\mean{a}_\rho-\mean{a}_{\eta\times I}} &\leq
    \frac{\norm{\rho}_{L^\infty(Q)}}{\norm{\rho}_{L^1(Q)}} \int_Q
    \bigabs{a-\mean{a}_{\eta\times I}} \dz \leq c_0 \dashint_Q
    \bigabs{a-\mean{a}_{\eta\times I}} \dz.
  \end{align*}
  In particular, we have
  \begin{align*}
    \dashint_{Q} \bigabs{a-\mean{a}_\rho}\dz &\leq (1+c_0) \dashint_Q
    \bigabs{a-\mean{a}_{\eta\times I}} \dz,
  \end{align*}
  so the general case follows from the special one.
\end{proof}
Since the above (weak) setting can not be applied to the Orlicz
setting in modular form, we include the following classical space-time
\Poincare{} in modular Orlicz form. 
\begin{lemma}
  \label{lem:poincare_phi}
  Let $Q=I \times B$ be $\alpha$-parabolic cube and let $\rho \in
  L^1(Q)$ be such that $\rho \geq 0$ and $\norm{\rho}_\infty \leq c_0
  \abs{Q}^{-1} \norm{\rho}_1$. 
  Moreover, let  $\partial_t
  a=\divergence G$ with $G\in L^1(Q)$ in the sense of distributions.  Let $\phi$ be an Orlicz
  function satisfying the~$\Delta_2$-condition.  Then for every
  $\alpha$-parabolic cube $Q = I \times B$ we have
  \begin{align*}
    \dashint_Q \phi\Big(\Bigabs{ \frac{a- \mean{a}_{\rho}}{r}
    }\Big) \dz &\leq c\, \dashint_Q \phi\big(\abs{\nabla a} \big)\dz +
    c\,\phi\bigg( \alpha \dashint_{Q}\abs{G}\dz\bigg).
  \end{align*}
\end{lemma}
\begin{proof}
  As in Theorem~\ref{thm:poincare_par} we begin with~$\rho = \chi_I
  \eta$ with~$\eta$ as in Lemma~\ref{lem:poincare_time}. Analogously
  to the proof of Theorem~\ref{thm:poincare_par} we estimate
  \begin{align*}
    \lefteqn{\dashint_{Q}\phi\Big(
      \Bigabs{\frac{a-\mean{a}_{\eta\times I}}{r}}\Big)\dz} \qquad &
    \\
    &\leq \dashint_{I} \dashint_{B}\phi\Big(
    \Bigabs{\frac{a-\mean{a(t)}_{\eta}}{r}}\Big)\dx \dt + \dashint_{I}
    \phi\Big(\Bigabs{\frac{\mean{a(t)}_{\eta}- \mean{a}_{\eta\times I}
      }{r}}\Big)\dt
    \\
    &=: I + II.
  \end{align*}
  Now $I$ can be estimated by $\dashint_Q \phi(\abs{\nabla a})\dz$ by
  using \Poincare{} in space for Orlicz functions,
  see e.g.~\cite[Theorem~7]{DieE08}. For the second we estimate
  \begin{align}
    \label{eq:time}
    \begin{aligned}
      \abs{\mean{a(t)}_{\eta}- \mean{a}_{\eta\times I}}&=
      \Bigabs{\dashint_I\mean{a(t)}_{\eta}-\mean{a(s)}_{\eta}\ds}
      =\Bigabs{\dashint_I\frac{1}{\norm{\eta}_{L^1(B)}}\int_s^t\skp{\partial_t
          a(\tau)}{\eta} d\tau \ds}
      \\
      &=\Bigabs{\dashint_I\frac{1}{\norm{\eta}_{L^1(B)}}\int_s^t\skp{G}{\nabla
          \eta}\ds}
      \\
      &\leq c\alpha r \dashint_Q \abs{G}\dz.
    \end{aligned}
  \end{align}
  This can be used to estimate $(II)$ and the claim follows for~$\rho=
  \chi_I \eta$.  

  Now as in the proof of Theorem~\ref{thm:poincare_par} we can change
  to general~$\rho$ by showing in the same manner
  \begin{align*}
    \dashint_{Q} \phi(\bigabs{a-\mean{a}_\rho})\dz &\leq (1+c_0) \dashint_Q
    \phi\big(\bigabs{a-\mean{a}_{\eta\times I}}\big) \dz. \qedhere
  \end{align*}
\end{proof}

\subsection{Extension}
\label{ssec:extension}

It is convenient for our purpose to use function which are defined on
the whole space~$\setR \times \setR^m$. Therefore, we will extend our
function~$w$ from~\eqref{eq:w} to a function on~$\setR \times
\setR^m$ such that most of its properties are preserved.

We therefore extend~$G$ and~$w$ from~$J \times \Omega$ to $(-\infty,0]
\times \setR^m$ by zero. Since $w(-t_0)=0$ in the sense of a
$(W^{1,\infty}_0(\Omega))^*)$-trace, it is easy to see that $\partial_t
w = \divergence G$ on $\mathcal{D}'((-\infty,0), \setR^m)$.

Next, we extend~$w$ to $\setR \times \setR^m$ by even reflection
and~$G$ by odd reflection. Then it follows that
\begin{align}
  \label{eq:wwhole}
  \begin{alignedat}{2}
    \partial_t w &= \Div G &\qquad &\text{in $\mathcal{D}'(\setR \times
      \Omega$)}
    \\
    w &= 0 &\qquad &\text{outside of $(-t_0,t_0) \times \Omega$},
    \\
    G &= 0 &\qquad &\text{outside of $(-t_0,t_0) \times \Omega$}.
  \end{alignedat}
\end{align}
We will construct a Lipschitz truncation~$\wal$ of~$w$ on~$\setR
\times \setR^m$, which is zero outside of~$(-t_0,t_0) \times
\Omega$. The restriction of~$\wal$ back to~$J \times \Omega$ will then
provide the Lipschitz truncation for our Theorem~\ref{thm:orlicz}.

\subsection{Whitney covering}
\label{ssec:whitney}

For $\alpha, \lambda>0$ we define the bad set~$\Oal$ as
\begin{align}
  \label{eq:Oal2}
  \Oal := \set{\Ma(\nabla w)>\lambda} \cup \set{\alpha \Ma(G) >
    \lambda}.
\end{align}
Note that this differs slightly from the definition~\eqref{eq:Oal} in
the Theorem~\ref{thm:orlicz}, since we extend~$w$ and~$G$ partly by
reflection. This increase the maximal function $\Ma(\nabla w)$ and
$\Ma(G)$ but at most by a factor of two. Therefore, for the sake of
readability we prefer to work with~\eqref{eq:Oal2}. The result certainly
also holds for~\eqref{eq:Oal}.

According to \cite[Lemma~3.1]{DieRuzWol10} there exists an
$\alpha$-parabolic Whitney covering $\set{Q_j^{\alpha}}= \set{I_j \times B_j}$ of~$\Oal$ in the
following sense:
\begin{enumerate}[label={\rm (W\arabic{*})}, leftmargin=*]
\item\label{itm:whit1} $\bigcup_j\frac {1} {2} Q_j^{\alpha} \,=\, \Oal$,
\item\label{itm:whit2} for all $j\in \setN$ we have $8
  Q_j^{\alpha} \subset \Oal$ and $16 Q_j^{\alpha} \cap (\setR^{m+1}\setminus
  \Oal)\neq \emptyset$, 
\item \label{itm:whit3} if $ Q_j^{\alpha} \cap Q_k^{\alpha} \neq \emptyset
  $ then $ \frac 12 r_k\le r_j\leq 2\, r_k$,
\item \label{itm:whit3b}$\frac 14 Q_j^{\alpha} \cap \frac 14Q_k^{\alpha} =
  \emptyset$  for all $j \neq k$,
\item \label{itm:whit4}  each $x\in  \Oal $ belongs to at most 
  $120^{m+2}$ of the sets $4Q_j^{\alpha}$,
\end{enumerate}
where $r_j := r_{B_i}$, the radius of $B_j$ and $Q_j^{\alpha} = I_j \times B_j$.

\noindent With respect to the covering $\set{Q_j^{\alpha}}$ there exists a partition of
unity $\set{\rho_j} \subset C^\infty_0(\setR^{m+1})$ such that
\begin{enumerate}[label={\rm (P\arabic{*})}, leftmargin=*]
\item \label{itm:P1} $\chi_{\frac{1}{2}Q_j^{\alpha}}\leq \rho_j\leq
  \chi_{\frac 34 Q_j^{\alpha}}$ 
\item \label{itm:P3} $\norm{\rho_j}_\infty + r_j \norm{\nabla
    \rho_j}_\infty + r_j^2 \norm{\nabla^2 \rho_j}_\infty + \alpha\,
  r_j^2 \norm{\partial_t \rho_j}_\infty \leq c$.
\end{enumerate}
For each $k \in \setN$ we define $A_k:= \set{ j \,:\, \frac 34
  Q_k^{\alpha} \cap \frac 34 Q_j^{\alpha} \neq \emptyset}$. Then
\begin{enumerate}[label={\rm (P\arabic{*})}, leftmargin=*,start=3]
\item \label{itm:P4} $\sum_{j \in A_k} \rho_j = 1$ on $\frac 34 Q_k^{\alpha}$.
\end{enumerate}
We get the following additional property
\begin{enumerate}[label={\rm (W\arabic{*})}, leftmargin=*, start=6]
\item \label{itm:whit_fat} If $j \in A_k$, then $\abs{Q_j^{\alpha}
    \cap Q_k^{\alpha}} \geq 16^{-m-2} \max \set{\abs{Q_j^{\alpha}},
    \abs{Q_k^{\alpha}}}$.
\item \label{itm:whit_fat34} If $j \in A_k$, then $\abs{\frac 34 Q_j^{\alpha}
    \cap \frac 34 Q_k^{\alpha}} \geq  \max \set{\abs{Q_j^{\alpha}},
    \abs{Q_k^{\alpha}}}$.
\item \label{itm:whit_radii} If $j \in A_k$, then $\frac 12 r_k \leq
  r_j < 2 r_k$.
\item \label{itm:whit_sum} $\# A_k \leq 120^{m+2}$.
\end{enumerate}
Now, we define $w_j$ by
\begin{align*}
  w_j^{\alpha}&:=
  \begin{cases}
    \mean{w}_{\rho_j} &\qquad \text{if $\frac 34 Q_j^{\alpha}
      \subset J \times \Omega$},
    \\
    0 &\qquad \text{else.}
  \end{cases}
\end{align*}

We define our truncation $\wal$ via the formula
\begin{align}
  \label{eq:def_ulambda}
  \wal &:= w - \sum_j \rho_j (w - w_j^{\alpha}).
\end{align}
Since the $\rho_j$ are locally finite, the sum is pointwise well
defined. We will see later that the sum converges also as a
distribution and in a few function spaces.

Note that the sum $\sum_j \rho_j (w - w_j^{\alpha})$ is zero outside
of~$(-t_0,t_0) \times \Omega$. So also~$\wal$ is zero outside of
$(-t_0,t_0) \times \Omega$. In fact, we have
\begin{align}
  \label{eq:supp_bad_parts}
  \support(\rho_j (w - w_j^{\alpha})) \subset \tfrac 34 Q_j^{\alpha}
  \cap ((-t_0,t_0) \times \Omega).
\end{align}
Indeed, $\support \rho_j \subset \frac 34 Q_j^{\alpha}$, so the case
$\frac 34 Q_j^{\alpha} \subset J \times \Omega$ is obvious. If $\frac
34 Q_j^{\alpha} \not\subset J \times \Omega$, then $w_j^{\alpha}=0$
and the claim follows by $\support \rho_j \subset \frac 34
Q_j^{\alpha}$ and $\support w \subset J \times \Omega$.

\subsection{Estimates on the Whitney cylinders}
\label{ssec:est_whitney}

We need a few auxiliary results that allow to estimate~$w-w_j^\alpha$
on our Whitney cylinders. The estimates are based on our parabolic
\Poincare{}'s inequality of subsection~\ref{ssec:poincare}.

Since the equation~$\partial_t w = \divergence G$ only holds on $\setR
\times \Omega$, we need the following auxiliary result to deal with
the case of cylinders that our also outside of this domain. We use the
fact that~$w$ is zero outside of~$\setR \times \Omega$.
\begin{lemma}
  \label{lem:N_Qout}
  Let $Q$ be an $\alpha$-parabolic cylinder with radius~$r$. If $\frac
  45 Q \not\subset \setR \times \Omega$, then
  \begin{align*}
    \alpha\mathcal{N}_{Q}(\partial_t w) &\leq c\,\mathcal{
      M}_{Q}(\nabla w).
  \end{align*}
\end{lemma}
\begin{proof}
  We calculate
  \begin{align*}
    \alpha\,\mathcal{N}_Q(\partial_t w) &= \alpha \sup_{\xi
      \in \mathcal{F}_Q} \big( r \abs{Q}^{-1}
    \abs{\skp{w}{\partial_t \xi}} \big)\leq c\, \dashint_Q
    \frac{\abs{w}}{r}\dz.
  \end{align*}
  Let $Q =: I \times B$.  Since $\frac 45 Q \not\subset \setR \times
  \Omega$ and~$\Omega$ has fat complement, we have $\abs{B
    \setminus \Omega} \geq c\, \abs{B}$. Thus, we can apply the space
  \Poincare{} (with $w=0$ outside of $\setR \times \Omega$) to get
  \begin{align*}
    \dashint_{Q} \frac{\abs{w}}{r}\dz \leq c\, \dashint_{Q}
    \abs{\nabla w}\dz.
  \end{align*}
  This proves the claim.
\end{proof}

\begin{lemma}
  \label{lem:est_w-uj}
  The following holds.
  \begin{enumerate}
  \item \label{itm:est_w-uj_in} If $\frac 34 Q_j^{\alpha} \subset J \times
    \Omega$, then $w_j^{\alpha} = \mean{w}_{\eta_j \times I}$ and
    \begin{align*}
      \dashint_{\frac 34 Q_j^{\alpha}} \biggabs{ \frac{w-w_j^{\alpha}}{r_j}} \dz &\leq
      c\, \mathcal{M}_{\frac 34 Q_j^{\alpha}}(\nabla w) + c\, \alpha
      \mathcal{N}_{\frac 34 Q_j^{\alpha}}(\partial_t w).
    \end{align*}
  \item \label{itm:est_w-uj_in2} If $\frac 45 Q_j^{\alpha} \subset \setR \times
    \Omega$ and $\frac 34 Q_j^{\alpha} \not\subset J \times \Omega$, then
    $w_j^{\alpha}=0$ and
    \begin{align*}
      \dashint_{\frac 45 Q_j^{\alpha}} \biggabs{ \frac{w-w_j^{\alpha}}{r_j}} \dz &\leq
      c\, \mathcal{M}_{\frac 45 Q_j^{\alpha}}(\nabla w) + c\, \alpha
      \mathcal{N}_{\frac 45 Q_j^{\alpha}}(\partial_t w).
    \end{align*}
  \item \label{itm:est_w-uj_out}If $\frac 45 Q_j^{\alpha} \not\subset \setR
    \times \Omega$, then $w_j^{\alpha}=0$ and
    \begin{align*}
      \dashint_{Q_j^{\alpha}} \biggabs{ \frac{w-w_j^{\alpha}}{r_j}} \dz &\leq c\,
      \mathcal{M}_{Q_j^{\alpha}}(\nabla w).
    \end{align*}
  \end{enumerate}
\end{lemma}
\begin{proof}
  Part~\ref{itm:est_w-uj_in} follows immediately from
  Theorem~\ref{thm:poincare_par} with~$\rho=\rho_j$.

  Let us consider part~\ref{itm:est_w-uj_in2}. In this situation
  $\setR \times \Omega \setminus (J \times \Omega)$ contains a large
  part of~$\frac 45 Q_j^\alpha$ so that we can find a a function~$\rho
  \in L^\infty$ with support in~$\frac 45 Q_j^\alpha \cap ((\setR
  \times \Omega) \setminus (J \times \Omega))$ such
  that~$\norm{\rho}_\infty \leq c\, \abs{Q_j^\alpha}^{-1}
  \norm{\rho}_1$. Since~$w=0$ on $\support(\rho)$, we have $w_j^\alpha
  = 0 = \mean{w}_\rho$. Again the claim follows by
  Theorem~\ref{thm:poincare_par}.

  \noindent 
  Let us now prove~\ref{itm:est_w-uj_out}. Since~$\frac 45 Q_j^\alpha
  \not\subset \setR \times \Omega$, we can find a function~$\rho$ with
  support outside in~$Q_j^\alpha \cap \setR \times \Omega$ with
  $\norm{\rho}_\infty \leq c\, \abs{Q_j^\alpha}^{-1}
  \norm{\rho}_1$. Since~$w=0$ on $\support(\rho)$, we have $w_j^\alpha
  = 0 = \mean{w}_\rho$. Now Theorem~\ref{thm:poincare_par} proofs our
  claim with an additional $\mathcal{N}_{Q_j^\alpha}(\partial_t w)$ on
  term on the right hand side. Due to Lemma~\ref{lem:N_Qout} this term
  can be controlled again by~$M_{Q_j^\alpha}(\nabla w)$, which proves
  our claim.
\end{proof}
\begin{lemma}
  \label{lem:w-w_j_lambda}
  We have
  \begin{align*}
    \dashint_{\frac 34 Q_j^{\alpha}} \biggabs{
      \frac{w-w_j^{\alpha}}{r_j}} \dz &\leq \dashint_{Q_j^{\alpha}}
    \abs{\nabla w} \dz + \alpha \mathcal{N}_{\frac 45
      Q_j^\alpha}(\partial_t w) + \alpha \dashint_{Q_j^{\alpha}}
    \abs{G} \dz \leq c\, \lambda.
  \end{align*}
\end{lemma}
\begin{proof}
  Since $16Q_j^{\alpha} \cap (\setR^{m+1} \setminus \Oal) \neq
  \emptyset$, it follows that $\mathcal{M}_{16Q_j^{\alpha}}(\nabla w)
  \leq \lambda$ and $\alpha\mathcal{M}_{16Q_j^{\alpha}}(G) \leq
  \lambda$. Thus also $\mathcal{M}_{Q_j^{\alpha}}(\nabla w) \leq
  c\,\lambda$ and $\alpha \mathcal{M}_{Q_j^{\alpha}}(G) \leq
  c\,\lambda$. 

  The estimate $\alpha \mathcal{N}_{Q_j^\alpha}(\partial_t w) \leq
  c\,\lambda$ follows from Remark~\ref{rem:NvsMdiv} if $\frac 45
  Q_j^\alpha \subset \setR \times \Omega$ and from
  Lemma~\ref{lem:N_Qout} if $\frac 45 Q_j^\alpha \not\subset \setR
  \times \Omega$.
\end{proof}
\begin{lemma}
  \label{lem:w-w_j_lambda_phi}
  We have
  \begin{align*}
    \dashint_{\frac 34 Q_j^{\alpha}} \phi\bigg( \biggabs{
      \frac{w-w_j^{\alpha}}{r_j}} \bigg) \dz &\leq \dashint_{Q_j^{\alpha}}
    \phi(\abs{\nabla w}) \dz + \phi \Bigg( \alpha \dashint_{Q^\alpha_j} \abs{G}\,dx \Bigg).
  \end{align*}
\end{lemma}
\begin{proof}
  The proof is similar to the one of Lemma~\ref{lem:est_w-uj} and
  Lemma~\ref{lem:w-w_j_lambda} by using Lemma~\ref{lem:poincare_phi}
  instead of Theorem~\ref{thm:poincare_par}.
\end{proof}

\subsection{Stability}
\label{sec:stability}

In this subsection we will show the stability of the Lipschitz
truncation with respect to some norms.
\begin{lemma}
  \label{lem:stabL1}
  If $w\in L^1(J, W^{1,1}_0(\Omega))$ and $G\in L^1(J\times \Omega)$,
  then $\wal\in L^1(J, W^{1,1}_0(\Omega))$. 
  Moreover,
  \begin{align*}
    \int_{ \Oal } \abs{\nabla (w-\wal)} \dz\leq c\int_{\Oal}\abs{\nabla
      w}\dz+ \lambda \abs{\Oal}.
  \end{align*}
\end{lemma}
\begin{proof}
  It follows from the definition of~$\wal$ that
  \begin{align*}
    w - \wal &= \sum_j \rho_j (w - w^\alpha_j).
  \end{align*}
  Due to~\eqref{eq:supp_bad_parts} the sum is zero outside of $\Oal$. Using
  that $\sum_j \rho_j =1$ on $\Oal$ we get
  \begin{align}
    \label{eq:nablawwal}
    \nabla(w - \wal) &= \nabla w + \sum_j \nabla \rho_j (w - w^\alpha_j).
  \end{align}
  Now it follows with the help of~\ref{itm:P3}, \ref{itm:whit1},
  \ref{itm:whit4} and \eqref{eq:supp_bad_parts} that
  \begin{align*}
    \int_{ \Oal } \abs{\nabla (w-\wal)}\dz &\leq \int_{\Oal} \abs{
      \nabla w} \dz +
    c\sum_k\int_{\frac34Q_j^{\alpha}}\Bigabs{\frac{w-w_j^{\alpha}}{r_j}}
    \dz.
  \end{align*}
  This and Lemma~\ref{lem:w-w_j_lambda} implies
  \begin{align*}
    \int_{ \Oal } \abs{\nabla (w-\wal)}\dz &\leq \int_{\Oal} \abs{
      \nabla w} \dz+ c\, \lambda \abs{\Oal},
  \end{align*}
  which proves the lemma.
\end{proof}

\begin{lemma} 
  \label{lem:stabLphi}
 We get
  \begin{align*}
    \int_{ \Oal } \phi(\abs{\nabla (w-\wal)}) \,dz \leq
    c\int_{\Oal}\phi(\abs{\nabla w})+ c\, \abs{\Oal}
    \phi(\lambda).
  \end{align*}
\end{lemma}

\begin{proof}
  The proof is similar to Lemma~\ref{lem:stabL1}.  Starting
  with~\eqref{eq:nablawwal} and using ~\ref{itm:P3}, \ref{itm:whit1},
  \ref{itm:whit4}, \eqref{eq:supp_bad_parts}, and
  the~$\Delta_2$-condition we get
  \begin{align*}
    \int_{ \Oal } \phi(\abs{\nabla (w-\wal)}) \dz
    &\leq \int_{\Oal} \phi(\abs{ \nabla w}) \dz+
    c\sum_k\int_{\frac34Q_j^{\alpha}}\phi\bigg(
    \Bigabs{\frac{w-w_j^{\alpha}}{r_j}} \bigg) \dz.
  \end{align*}
  By Lemma~\ref{lem:w-w_j_lambda_phi} we
  can estimate the summands of the second part by
  \begin{align*}
    \int_{\frac34Q_j^{\alpha}}\phi\bigg(
    \Bigabs{\frac{w-w_j^{\alpha}}{r_j}} \bigg) \dz &\leq c\,
    \int_{\frac34Q_j^{\alpha}}\phi(\abs{\nabla w}) \dz + c\,
    \abs{Q^\alpha_j} \phi\bigg( \alpha
    \dashint_{\frac34Q_j^{\alpha}} \abs{G}\,dz \bigg).
  \end{align*}
  Using Lemma~\ref{lem:w-w_j_lambda}, we see that the mean value
  integral in the last term is bounded by~$c\,\lambda$, so overall we get
  \begin{align*}
    \int_{ \Oal } \phi(\abs{\nabla (w-\wal)}) \dz
    &\leq \int_{\Oal} \phi(\abs{ \nabla w}) \dz+
    c\, \abs{\Oal} \phi(\lambda),
  \end{align*}
  which proves the lemma.
\end{proof}

\section{Lipschitz property}
\label{sec:lipschitz-property}

In this section we show that the truncated function~$\wal$ has some
sort of Lipschitz properties. In particular, we used
$\mathcal{M}^\alpha(\nabla w)$ and
$\alpha\,\mathcal{N}^\alpha(\partial_t w)$ (more precisely its upper
bound $\mathcal{M}^\alpha(G)$) to define the bad set, where we
truncate the function. 
It turns out  that
$\mathcal{M}^\alpha(\nabla \wal) + \alpha
\mathcal{N}^\alpha(\partial_t \wal) \leq c\, \lambda$.
\begin{lemma}
  \label{lem:diff_uj_uk}
  \begin{align*}
    \sum_{j \in A_k} \frac{\abs{w_j^{\alpha}-w_k^{\alpha}}}{r_j} &\leq c\, \sum_{j \in
      A_k} \dashint_{\frac 34 Q_j^{\alpha}} \frac{\abs{w-w_j^{\alpha}}}{r_j} \dz \leq
    c\, \lambda.
  \end{align*}
\end{lemma}
\begin{proof}
  Due to~\ref{itm:whit_fat34} and~\ref{itm:whit_radii} for every $j
  \in A_k$ holds $\abs{\frac 34 Q_j^{\alpha} \cap \frac 34 Q_k^{\alpha}} \geq   \max
  \set{\abs{Q_j^{\alpha}}, \abs{Q_k^{\alpha}}}$ and $r_j \geq \frac 12 r_k$. Thus we can
  estimate
  \begin{align*}
    \sum_{j \in A_k} \frac{\abs{w_j^{\alpha}-w_k^{\alpha}}}{r_j} &\leq \sum_{j \in A_k}
    \dashint_{Q_j^{\alpha} \cap Q_k^{\alpha}} \frac{\abs{w_j^{\alpha}-w_k^{\alpha}}}{r_j} \dz
    \\
    &\leq c\, \dashint_{Q_j^{\alpha} \cap Q_k^{\alpha}} \frac{\abs{w-w_k^{\alpha}}}{r_k} \dz + c
    \sum_{j \in A_k} \dashint_{Q_j^{\alpha} \cap Q_k^{\alpha}} \frac{\abs{w-w_j^{\alpha}}}{r_j}
    \dz
    \\
    &\leq \sum_{j \in A_k} \dashint_{Q_j^{\alpha}} \frac{\abs{w-w_j^{\alpha}}}{r_k} \dz,
  \end{align*}
  where we also used $k \in A_k$. The rest follows by
  Lemma~\ref{lem:w-w_j_lambda}. 
\end{proof}
We need the following geometric alternatives.
\begin{lemma}
  \label{lem:alternatives}
  Let $Q$ be an $\alpha$-parabolic cylinder with radius~$r$. Then at
  least one of the following alternatives holds.
  \begin{enumerate}[label={(A\arabic{*})}]
  \item \label{itm:alt_small} There exists $k \in \setN$ such that $Q \cap
    \frac 12 Q_k^{\alpha} \neq \emptyset$, $8r \leq r_k$ and $Q \subset \frac
    34 Q_k^{\alpha}$.
  \item \label{itm:alt_large} For all $j \in\setN$ with $Q \cap \frac
    34 Q_j^{\alpha} \neq \emptyset$, there holds $r_j \leq 16 r$ and $\abs{Q_j^{\alpha}}
    \leq 8^{m+2} \abs{Q_j^{\alpha} \cap Q}$. Moreover, $137 Q \cap
    (\setR^{m+1}  \setminus \Oal) \neq \emptyset$.
  \end{enumerate}
\end{lemma}
\begin{proof}
  If there exists $k \in \setN$ such that $Q \cap \frac 12 Q_k^{\alpha} \neq
  \emptyset$ and $8r \leq r_k$, then automatically $Q \subset Q_k^{\alpha}$.
  Assume now that such an~$k$ does not exist. Then for every $l \in
  \setN$ with $Q \cap \frac 12 Q_l \neq \emptyset$, there holds $r_l
  \leq 8r$. Suppose that $Q \cap \frac 34 Q_j^{\alpha} \neq \emptyset$. Now let
  $x \in Q \cap \frac 34 Q_j^{\alpha}$, then by~\ref{itm:whit1} there exists
  $m$ such that $x \in \frac 12 Q_m$. In particular, we have $Q \cap
  \frac 34 Q_j^{\alpha} \neq \emptyset$ and $\frac 12 Q_m \cap \frac 34 Q_j^{\alpha}
  \neq \emptyset$, since both sets contain~$x$. Now, our assumption
  and $Q \cap \frac 34 Q_j^{\alpha} \neq \emptyset$ implies $r_m \leq 8r$. On
  the other hand $\frac 12 Q_m \cap \frac 34 Q_j^{\alpha} \neq \emptyset$
  and~\ref{itm:whit3} imply $r_j \leq 2 r_m$. Thus, $r_j \leq 16 r$.
  Moreover, it follows from $8r \geq r_m$ that $137 Q = (1+17\cdot 8)Q
  \supset 16 Q_m$.  Since $16 Q_m \cap (\setR^{m+1} \setminus \Oal)
  \neq \emptyset$, we also get $137 Q \cap (\setR^{m+1} \setminus
  \Oal) \neq \emptyset$. Now, let $z_0 \in Q \cap \frac 34 Q_j^{\alpha}$. It
  remains to prove $\abs{Q_j^{\alpha}} \leq 8^{m+2} \abs{Q_j^{\alpha} \cap Q}$. If $r
  \leq \frac 18 r_j$, then $Q \subset Q_j^{\alpha}$ and the claim follows. If
  $r \geq \frac 18 r_j$, then there exists an $\alpha$-parabolic
  cylinder $Q'$ with radius $\frac 18 r_j$ such that $Q' \subset Q_j^{\alpha}
  \cap Q$.  So in this case $\abs{Q_j^{\alpha} \cap Q} \geq \abs{Q'} \geq
  8^{-m-2} \abs{Q_j^{\alpha}}$.
\end{proof}
\begin{lemma}
  \label{lem:lip-nablaw}
  There holds
  \begin{align*}
    \Ma(\nabla \wal) &\leq
    c\,\lambda.
  \end{align*}
\end{lemma}
\begin{proof}
  Let $Q$ be an $\alpha$-parabolic cylinder with radius~$R$. We use
  the alternatives of Lemma~\ref{lem:alternatives}.

  We begin with alternative~\ref{itm:alt_small}.  In
  particular, there exists $k \in \setN$ such that $Q \cap \frac 12
  Q_k^{\alpha} \neq \emptyset$, $8 R \leq r_k$ and $Q \subset \frac 34
  Q_k^{\alpha}$.

\noindent  Then $w = \sum_{j \in A_k} \rho_j w_j^{\alpha}$ on $Q$ and therefore
  \begin{align*}
    \mathcal{M}_Q(\nabla \wal) &= \mathcal{M}_Q(\nabla
    (\wal-w_k^{\alpha})) = \mathcal{M}_Q \bigg( \nabla \Big(\sum_{j \in
      A_k} \rho_j(w_j^{\alpha} -w_k^{\alpha}) \Big)\bigg)
    \\
    &\leq \sum_{j \in A_k} \mathcal{M}_Q \Big( \nabla \big( \rho_j(w_j^{\alpha} -w_k^{\alpha})
    \big) \Big)
    \\
    &\leq c\, \sum_{j \in A_k} \frac{\abs{w_j^{\alpha} -w_k^{\alpha}}}{r_j}.
  \end{align*}
  Now, Lemma~\ref{lem:diff_uj_uk} implies $\mathcal{M}_Q(\nabla
  \wal) \leq c\,\lambda$.

\noindent  We turn to alternative~\ref{itm:alt_large}. In particular, for all
  $j \in\setN$ with $Q \cap \frac 34 Q_j^{\alpha} \neq \emptyset$, there holds
  $r_j \leq 16 r$ and $\abs{Q_j^{\alpha}} \leq 8^{m+2} \abs{Q_j^{\alpha} \cap Q}$.
  Moreover, $137 Q \cap (\setR^{m+1} \setminus \Oal) \neq \emptyset$.
  Using $\wal = w - \sum_j \rho_j (w-w_j^{\alpha})$ we estimate
  \begin{align*}
    \mathcal{M}_Q(\nabla \wal) &\leq \mathcal{M}_Q(\nabla w) +
    \sum_{j\,:\, Q \cap \frac 34 Q_j^{\alpha} \neq \emptyset}
    \mathcal{M}_Q(\nabla(\rho_j(w-w_j^{\alpha})).
    \\
    &\leq \mathcal{M}_Q(\nabla w) + c \!\!\!\sum_{j\,:\, Q \cap \frac
      34 Q_j^{\alpha} \neq \emptyset} \frac{\abs{Q_j^{\alpha}}}{\abs{Q}}
    \mathcal{M}_{\frac 34 Q_j^{\alpha}}(\nabla(\rho_j(w-w_j^{\alpha}))).
    \\
    &\leq \mathcal{M}_Q(\nabla w) + c \!\!\!\sum_{j\,:\, Q \cap \frac
      34 Q_j^{\alpha} \neq \emptyset} \frac{\abs{Q_j^{\alpha} \cap Q}}{\abs{Q}}
    \mathcal{M}_{\frac 34 Q_j^{\alpha}}(\nabla(\rho_j(w-w_j^{\alpha}))).
  \end{align*}
  Due to Lemma~\ref{lem:w-w_j_lambda} there holds
  \begin{align*}
    \mathcal{M}_{\frac 34 Q_j^{\alpha}}(\nabla(\rho_j(w-w_j^{\alpha}))) &\leq c
    \dashint_{\frac 34 Q_j^{\alpha}} \frac{\abs{w-w_j^{\alpha}}}{r_j}\dz +
    c\dashint_{\frac 34 Q_j^{\alpha}} \abs{\nabla w}\dz \leq c\,\lambda.
  \end{align*}
  On the other hand
  \begin{align*}
    \mathcal{M}_Q(\nabla w) \leq \mathcal{M}_{137Q}(\nabla w) \leq c\,
    \lambda,
  \end{align*}
  since $137 Q \cap (\setR^{m+1} \setminus \Oal) \neq \emptyset$. We
  summarize the above estimate to get
  \begin{align*}
    \mathcal{M}_Q(\nabla \wal) &\leq c\, \lambda + c\, \lambda \sum_{j\,:\, Q
      \cap \frac 34 Q_j^{\alpha} \neq \emptyset} \frac{\abs{Q_j^{\alpha} \cap
        Q}}{\abs{Q}}  \leq c\, \lambda,
  \end{align*}
  where we used that the $Q_j^{\alpha}$ are locally finite, see~\ref{itm:whit4}.
\end{proof}

\begin{lemma}
  \label{lem:lip-wt}
  There holds
  \begin{align*}
    \alpha\, \Na(\partial_t \wal) \leq c\,\lambda.
  \end{align*}
\end{lemma}
\begin{proof}
  Let $Q$ be an $\alpha$-parabolic cylinder with radius~$R$ and~$Q
  \subset \setR \times \Omega$. We have to show that $\alpha
  \mathcal{N}_Q (\partial_t \wal) \leq c\, \lambda$.
  If~$137 Q \not\subset \setR \times \Omega$, then the claim follows
  from Lemma~\ref{lem:N_Qout} and Lemma~\ref{lem:lip-nablaw}, so we
  can assume in the following $137 Q \subset \setR \times \Omega$.
  We use the alternatives of Lemma~\ref{lem:alternatives}.

  We begin with alternative~\ref{itm:alt_small}.  In particular, there
  exists $k \in \setN$ such that $Q \cap \frac 12 Q_k \neq \emptyset$,
  $8R \leq r_k$ and $Q \subset \frac 34 Q_k$.

  Then $\wal = \sum_{j \in A_k} \rho_j w^\alpha_j$ on $Q_k$ and
  therefore
  \begin{align*}
    \alpha \mathcal{N}_Q(\partial_t \wal) &= \alpha
    \mathcal{N}_Q(\partial_t (\wal-w^\alpha_k)) = \alpha \mathcal{N}_Q
    \bigg(
    \partial_t \Big(\sum_{j \in A_k} \rho_j(w^\alpha_j -w^\alpha_k) \Big)\bigg)
    \\
    &= \sum_{j \in A_k} \alpha \mathcal{N}_Q\big( \partial_t (\rho_j
    (w^\alpha_j-w^\alpha_k)) \big).
  \end{align*}
  We estimate
  \begin{align*}
    \alpha \mathcal{N}_Q\big( \partial_t (\rho_j (w^\alpha_j-w^\alpha_k)) \big) &=
    \alpha \sup_{\xi \in \mathcal{F}_Q} \Big( R\, \abs{Q}^{-1}
    \abs{\skp{\partial_t(\rho_j(w^\alpha_j-w^\alpha_k))}{\xi}} \Big)
    \\
    &\leq \alpha \sup_{\xi \in \mathcal{F}_Q} \Big( R\,
    \norm{\partial_t \rho_j}_\infty \abs{w^\alpha_j - w^\alpha_k} \norm{\xi}_\infty
    \Big)
    \\
    &\leq c\, \alpha R\, \abs{w^\alpha_j - w^\alpha_k} (\alpha r_j^2)^{-1}
    \\
    &\leq c\, \frac{\abs{w^\alpha_j - w^\alpha_k}}{r_j},
  \end{align*}
  where we used $8R \leq r_k \leq 2 r_j$ in the last step. This and
  Lemma~\ref{lem:diff_uj_uk} imply $\alpha \mathcal{N}_Q(\partial_t
  \wal) \leq c\, \lambda$.

  We turn to alternative~\ref{itm:alt_large}. In particular, for all
  $j \in\setN$ with $Q \cap \frac 34 Q_j \neq \emptyset$, there holds
  $r_j \leq 16 r$ and $\abs{Q_j} \leq 8^{d+2} \abs{Q_j \cap Q}$.
  Moreover, $137 Q \cap (\setR^{d+1} \setminus \Oal) \neq \emptyset$.
  Using $\wal = w - \sum_j \rho_j (w-w^\alpha_j)$ we estimate
  \begin{align*}
    \mathcal{N}_Q(\partial_t \wal) &\leq \mathcal{N}_Q(\partial_t
    w) + \sum_{j\,:\, Q \cap \frac 34 Q_j \neq \emptyset}
    \mathcal{N}_Q(\partial_t(\rho_j(w-w^\alpha_j)).
  \end{align*}
  Recall that $137Q \subset \setR \times \Omega$. So $137 Q \cap
  (\setR^{d+1} \setminus \Oal) \neq \emptyset$ implies
  $\alpha\mathcal{N}_Q(\partial_t w) \leq c\, \alpha \mathcal{M}_Q(G)
  \leq c\,\lambda$ using also Remark~\ref{rem:NvsMdiv}.
  On the other hand using $r_j \leq 16 R$, Lemma~\ref{lem:diff_uj_uk}
  and $\abs{Q_j} \leq 8^{d+2} \abs{Q_j \cap Q}$ we estimate 
  \begin{align*}
    \mathcal{N}_Q(\partial_t \wal) &\leq c\,\lambda + \sum_{j: Q \cap
      \frac 34 Q_j \neq \emptyset} \alpha\,
    \mathcal{N}_Q\big(\partial_t(\rho_j(w-w^\alpha_j)\big)
    \\
    &= c\, \lambda + \alpha \sum_{j: Q \cap \frac 34 Q_j \neq
      \emptyset} \sup_{\xi \in \mathcal{F}_Q} \Big( R\, \abs{Q}^{-1}
    \abs{\skp{\rho_j(w-w^\alpha_j)}{\partial_t \xi}} \Big).
  \end{align*}
  Now for $j$ with $Q \cap \frac 34 Q_j \neq \emptyset$ and  $\xi_j :=
  \mean{\xi}_{Q_j}$ we have
  \begin{align*}
    \lefteqn{\frac{\alpha R}{\abs{Q}}
      \abs{\skp{\rho_j(w-w^\alpha_j)}{\partial_t \xi}}} \qquad &
    \\
    &\leq \frac{\alpha R}{\abs{Q}} \abs{\skp{w-w^\alpha_j}{ \partial_t \big(
        \rho_j (\xi-\xi_j)\big)}} + \frac{\alpha R}{\abs{Q}}
    \abs{\skp{w-w^\alpha_j}{(\partial_t \rho_j)(\xi-\xi_j)}}
    \\
    &=: I + II.
  \end{align*}
  We will now estimate
  $\norm{\rho_j(\xi-\xi_j)}_{\mathcal{F}_{Q_j}}$. Using
  $\norm{\xi}_\infty + R \norm{\nabla \xi}_\infty + \alpha R^2
  \norm{\partial_t \xi}_\infty\leq 1$, we get by parabolic
  \Poincare{}'s inequality
  \begin{align*}
    \norm{\xi-\xi_i}_{L^\infty(Q_j)} &\leq c\, r_i \norm{\nabla
      \xi}_{L^\infty(Q_j)} + c\, \alpha r_i^2 \norm{\partial_t
      \xi}_{L^\infty(Q_j)} \leq c\, \frac{r_i}{R} + c\,
    \frac{r_i^2}{R^2} \leq c\, \frac{r_i}{R},
    \\
    \norm{\rho_j(\xi-\xi_j)}_\infty &\leq
    \norm{\xi-\xi_i}_{L^\infty(Q_j)} \leq c\, \frac{r_i}{R},
    \\
    \norm{\nabla (\rho_j(\xi-\xi_j))}_\infty &\leq
    \norm{\xi-\xi_i}_{L^\infty(Q_j)} + r_i \norm{\nabla
      \xi}_{L^\infty(Q_j)} \leq c\, \frac{r_i}{R},
    \\
    \norm{\partial_t(\rho_j(\xi-\xi_j))}_\infty &\leq
    \norm{\xi-\xi_i}_{L^\infty(Q_j)} + \alpha r_i^2 \norm{\partial_t
      \xi}_{L^\infty(Q_j)} \leq c\, \frac{r_i}{R} + c\,
    \frac{r_i^2}{R^2} \leq c\, \frac{r_i}{R}.
  \end{align*}
  In particular, $\norm{\rho_j(\xi-\xi_j)}_{\mathcal{F}_{Q_j}}\leq c
  \frac{r_i}{R}$. This and Lemma~\ref{lem:w-w_j_lambda} imply
  \begin{align*}
    I &= \frac{\alpha R}{\abs{Q}} \abs{\skp{w-w^\alpha_j}{ \partial_t
        \big( \rho_j (\xi-\xi_j)\big)}}
    \\
    &= \frac{\alpha R}{\abs{Q}} \abs{\skp{w}{ \partial_t \big( \rho_j
        (\xi-\xi_j)\big)}}
    \\
    &\leq c\, \frac{\alpha R}{\abs{Q}} \frac{\abs{Q_i}}{r_i}
    \,\mathcal{N}_{Q_j}(\partial_t w)
    \norm{\rho_j(\xi-\xi_j)}_{\mathcal{F}_{Q_j}}
    \\
    &\leq c\, \frac{\alpha R}{\abs{Q}} \frac{\abs{Q_i}}{r_i}
    \,\frac{\lambda}{\alpha} \frac{r_i}{R}
    \\
    &= c\, \lambda \frac{\abs{Q_i}}{\abs{Q}}.
  \end{align*}
  Moreover, also by Lemma~\ref{lem:w-w_j_lambda}
  \begin{align*}
    II &= \frac{\alpha R}{\abs{Q}} \abs{\skp{w-w^\alpha_j}{(\partial_t
        \rho_j)(\xi-\xi_j)}}
    \\
    &\leq \frac{\alpha R}{\abs{Q}} \abs{Q_j} \dashint_{Q_j} \abs{w -
      w^\alpha_j}\dz \frac{c}{\alpha r_j^2}
    \norm{\xi-\xi_j}_{L^\infty(Q_j)}
    \\
    &\leq \frac{\alpha R}{\abs{Q}} \abs{Q_j}\, r_j \lambda  \frac{c}{\alpha
      r_j^2} \frac{r_j}{R}
    \\
    &= c\, \frac{\abs{Q_j}}{\abs{Q}} \lambda.
  \end{align*}
  Summarized we have
  \begin{align*}
    \alpha\, \mathcal{N}_Q(\partial_t\wal) &\leq c\,\lambda + \!\!\!\sum_{j:
      Q \cap \frac 34 Q_j \neq \emptyset } \!\!c\,
    \frac{\abs{Q_j}}{\abs{Q}} \lambda &\leq c\,\lambda + \!\!\!\sum_{j: Q
      \cap \frac 34 Q_j \neq \emptyset } \!\!c\, \frac{\abs{Q_j \cap
        Q}}{\abs{Q}} \lambda \leq c\,\lambda,
  \end{align*}
  This proves the claim.
\end{proof}

\begin{lemma}
  There holds
  \begin{align*}
    \Mas(\wal) \leq c\,\lambda.
  \end{align*}
\end{lemma}
\begin{proof}
  Due to Theorem~\ref{thm:poincare_par}, Lemma~\ref{lem:lip-nablaw}
  and Lemma~\ref{lem:lip-wt} we have 
  \begin{align*}
    M^{\sharp,1}_Q(w^\alpha_\lambda) &\leq c\, M_Q(\nabla
    w^\alpha_\lambda) + c\, \alpha\, \mathcal{N}_Q(\partial_t
    w^\alpha_\lambda) \leq c\, \lambda. 
  \end{align*}
  for every~$\alpha$-parabolic cylinder~$Q$.
\end{proof}

\begin{corollary}
  \label{cor:hoelder_ul}
  $\wal$ is Lipschitz continuous with respect to $d^\alpha$, i.e.
  \begin{align*}
    \abs{\wal(t,x) - \wal(s,y)} \leq c\, \lambda \max
    \biggset{ \frac{\abs{t-s}^\frac 12}{\alpha^{\frac 12}}, \abs{x-y}
    }
  \end{align*}
\end{corollary}
\begin{proof}
  It follows from $\Mas(\wal) \leq c\,\lambda$ and~\cite{DaP65} that
  $\wal$ is Lipschitz continuous with respect to $d^\alpha$.
\end{proof}
\begin{lemma} 
  \label{lem:integrationsbyparts}
  Let $J=(t^-,t^+)$. For all $\eta\in W^{1,\infty}_0(-\infty,t^+) $ the expression $\skp{\partial_t
    w}{\wal \eta}$ is well defined and can be calculated as
  \begin{align}
    \label{eq:intbyparts}
    \skp{\partial_t w}{\wal \eta} &= \frac12 \int_Q (|\wal |^2
    -2w\cdot \wal) \partial_t \eta\, dz+\int_{\Oal} (\partial_t \wal)
    (\wal-w)\eta\, dz.
  \end{align}
\end{lemma}
\begin{proof}
  Let $0<h<T$. For a function $f$ defined in space and time denote the
  Steklov average of~$f$ by
  \begin{align*}
    f_h(x,t):=\frac{1}{h}\int_t^{t+h} f(x,s)\, ds.
  \end{align*}
  Then we have $\partial_t f_h(x,t)=h^{-1} (f(x, t+h)-f(x,t))$. We
  calculate
  \begin{align*}
    (I)_h &:= \skp{\partial_t w}{((\wal)_h \eta)_{-h}} = -\int_Q
    w_h \cdot \partial_t \big( (\wal)_h \eta \big) dz
    \\
    &= \int_Q (\wal-w)_h \cdot \partial_t((\wal)_h \eta)\,dz - \int_Q
    (\wal)_h \cdot \partial_t((\wal)_h \eta )\dz
    \\
    &= \int_Q (\wal-w)_h \cdot \big(\partial_t(\wal)_h\big) \eta\,dz +
    \int_Q (\wal-w)_h \cdot (\wal)_h \partial_t \eta\,dz - \int_Q \tfrac 12
    \abs{(\wal)_h}^2 \partial_t \eta\dz
    \\
    &= \int_Q (\wal-w)_h \cdot \big(\partial_t(\wal)_h\big) \eta\,dz + \frac
    12 \int_Q \big(\abs{(\wal)_h)}^2 - 2 w_h \cdot(\wal)_h\big)\partial_t
    \eta\,dz
    \\
    &=: (II)_h + (III)_h.
  \end{align*}
  All of these expressions are well defined. It has been shown
  in~\cite{DieRuzWol10} formula (3.33) that 
  \begin{align*}
    (II)_h &\to \int_Q (\wal-w) \big(\partial_t \wal \big) \eta\,dz,
    \\
    (III)_h &\to \frac
    12 \int_Q \big(\abs{\wal}^2 - 2 w \cdot \wal \big)\partial_t
    \eta\,dz.
  \end{align*}
  for $h \to 0$. Let us point out that $\wal-w$ is only non-zero
  on~$\Oal$. On this set~$\wal$ is locally~$C^\infty$, so $\partial_t
  \wal$ is  a classical time derivative on this set. This shows that
  the limit~$(I)_h$ is also well defined and can be calculated
  by~\ref{eq:intbyparts}. 
\end{proof}

This was the last piece to get Theorem~\ref{thm:orlicz}. 


\begin{proof}[Proof of Theorem~\ref{thm:orlicz}]
The definition of Lipschitz truncation $\wal$ is given in \eqref{eq:def_ulambda} and property (a) follows by the definition.
Property (b)  is proven in Lemma \ref{lem:lip-nablaw}, property (c) is proven in Lemma \ref{lem:stabLphi}, property (d) is proven in Lemma \ref{lem:lip-wt}, property (e) is proven in Corollary \ref{cor:hoelder_ul}, property (f) follows by Lemma \ref{lem:integrationsbyparts}.
 
\end{proof}

\section{The $\phi$-caloric approximation}\label{phi}

In this Section we will concentrate to prove the $\phi$-caloric approximation result i.e. Theorem \ref{p-caloric} in the general case of $\varphi$-growth.

Let us start defining $\bfA,\bfV\,:\, \setR^{m \times n} \to \setR^{m \times n}$ in
the following way:
\begin{subequations}
  \label{eq:defAV}
  \begin{align}
    \label{eq:defA}
    \bfA(\bfQ)&=\phi'(|\bfQ|)\frac{\bfQ}{|\bfQ|},
    \\
    \label{eq:defV}
    \bfV(\bfQ)&=\psi'(|\bfQ|)\frac{\bfQ}{|\bfQ|}.
  \end{align}
\end{subequations}

Another important set of tools are the {\rm shifted N-functions}
$\set{\phi_a}_{a \ge 0}$. We define for $t\geq0$
\begin{align}
  \label{eq:phi_shifted}
  \phi_a(t):= \int _0^t \varphi_a'(s)\, ds\qquad\text{with }\quad
  \phi'_a(t):=\phi'(a+t)\frac {t}{a+t}.
\end{align}
Note that $\phi_a(t) \sim \phi'_a(t)\,t$. The families $\set{\phi_a}_{a \ge 0}$ and
$\set{(\phi_a)^*}_{a \ge 0}$ satisfy the $\Delta_2$-condition uniformly in $a \ge 0$. 
The connection between $\bfA$, $\bfV$ (see ~\cite{DieStrVer12})  is the following:
\begin{align*}
    \big({\bfA}(\bfP) - {\bfA}(\bfQ)\big) \cdot \big(\bfP-\bfQ \big)
    &\sim \bigabs{ \bfV(\bfP) - \bfV(\bfQ)}^2 \sim
    \phi_{\abs{\bfP}}(\abs{\bfP - \bfQ}),
    \\
      \intertext{uniformly in $\bfP, \bfQ \in \setR^{m \times n}$ .
      Moreover,} \bfA(\bfQ) \cdot \bfQ \sim \abs{\bfV(\bfQ)}^2 &\sim
    \phi(\abs{\bfQ}),
  \end{align*}
  uniformly in $\bfQ \in \setR^{m \times n}$.

Now we begin to prove some Lemmas regarding the level sets of the maximal function.   Let  $w\in L^\phi(J,W^{1,\phi}_0(\Omega))$ and  $G\in L^{\phi^*}(J\times\Omega)$  such that
\begin{align}
\label{label1}
\left\{ \quad \begin{array}{l}
      \partial_t w=\Div G \,,  \quad \text{on} \quad [-t_0,0)\times \Omega\\[0.1cm]
      w(-t_0,\cdot)\equiv 0 \,.
   \end{array} \right.
\end{align}
We define for $Q=[-t_0,0)\times\Omega$
  \begin{align}
\label{eq:gamma1}
    \phi(\gamma) &:= \dashint_Q \phi(\abs{\nabla w})\dz + \dashint_Q
    \phi^*(\abs{G})\dz.
  \end{align}
We then have the following lemma.
\begin{lemma}
\label{lem:goodlambda}
For every $m_0\in\setN$ there exists a  $\lambda\in [\gamma,2^{m_0}\gamma]$, such that for $\alpha=\alpha(\lambda):=\frac{\lambda}{\phi'(\lambda)}$
  \begin{align*}
    \abs{\set{\mathcal{M}^{\alpha}(\nabla w\chi_Q)>\lambda}}+\abs{\set{\mathcal{M}^\alpha(G\chi_Q)>\phi'(\lambda)}}
    \leq c\frac{\phi(\gamma)}{m_0\phi(\lambda)}\abs{Q}
  \end{align*}
with $c$ independent of $m_0$ and $\gamma$.
\end{lemma}
\begin{proof}
We will use the following maximal operator
\[
 \Mst(f)(z):=\sup_\set{I\times B\subset\setR^{m+1}\,:\,z\in I\times B} \dashint_I\dashint_B f \dx\dt.
\]
Certainly we have that $\Ma(f)(x)\leq \Mst(f)(x)$, for almost all $x\in \setR^{m+1}$. Therefore,
  \begin{align*}
    \Oal := \set{\Ma(\nabla w)>\lambda} \cup \set{\alpha \Ma(G) >
      \lambda} \subset \set{\Mst(\nabla w)>\lambda}\cup \set{\Mst(G)
      >\frac{\lambda}{\alpha}}.
  \end{align*}
  Now we have by the continuity of $\Mst$ and since
  $(\phi')^{-1}\sim(\phi^*)'$, that for $m_0\in\setN$ and
  $\alpha(t):=\frac{t}{\phi'(t)}$,
  \begin{align*}
    &m_0\min_{m\in\set{0,..,m_0}}\phi(2^m\gamma) \Big( \big(\abs{
      \set{\mathcal{M}^{\alpha} (\nabla
        w\chi_Q)>2^m\gamma}}+\abs{\set{\mathcal{M}^{\alpha}(G\chi_Q)
        >\phi'(2^m\gamma)}}\big) \Big)
    \\
    &\quad\leq \sum_{m=0}^{m_0} \Big(\phi(2^m\gamma)(\abs{\set{\mathcal{M}^{
          \alpha}(\nabla
        w\chi_Q)>2^m\gamma}}+\abs{\set{\mathcal{M}^{\alpha}
        (G\chi_Q)>\phi'(2^m\gamma)}} \Big)
    \\
    &\quad\leq \sum_{m=0}^{m_0} \Big( \phi(2^m\gamma)\abs{\set{\Mst(\nabla
        w\chi_Q)>2^m\gamma}}+\abs{\set{\Mst(G\chi_Q)>\phi'(2^m\gamma)}} \Big)
    \\
    &\quad\leq \int \phi(\Mst(\nabla
    w\chi_Q))+\phi((\phi')^{-1}(\Mst(G\chi_Q)))\dz
    \\
    &\quad \leq c\int_{Q}\phi(\abs{\nabla w})+\phi^*(\abs{G})
\leq
  c\phi(\gamma)\abs{Q}.
  \end{align*}
This concludes the proof.
\end{proof}

Let $ u\in L^\phi(J,W^{1,\phi}_0(\Omega))$  be solution of 
\begin{align*}
  \partial_t u &= \divergence H
\end{align*}
on $Q=I\times B=(t^-,t^+)\times B$
with $H\in L^{\phi^*}(J\times\Omega)$ 
and  $h$ be the weak solution of
\begin{align*}
  \partial_t h -  \divergence(A(\nabla h)) &= 0\text{ in }Q
\end{align*}
with $h=u$ on $\partial_p Q$. The function $h$ is called the $\phi$-caloric comparison function of $u$ in $Q$. Define $w := u-h$. Then
\begin{align*}
  \partial_t w -  \divergence(A(\nabla u)-A(\nabla h)) &=
  \partial_t u - \divergence(A(\nabla u))
  \\
  &= \divergence(H-A(\nabla u)) = \divergence(G)
\end{align*}
and $w=0$ on $\partial_p Q$, where $G= H-A(\nabla u)$.

Since $w$ is a valid testfunction, we find by the standard methods, that 
\begin{align}
\label{eq:linftyl2}
\sup_{t\in I}\dashint_B\frac{\abs{w}^2}{t^+-t^-}\dx+\seb{\dashint_{Q}\abs{V(\nabla u)-V(\nabla h)}^2\dz} \leq c_0\dashint_{Q} \phi(\abs{\nabla u})+\phi^*(\abs{G})\dz,
\end{align}
where 
$c_0$ is a fixed constant only depending on the characteristics of $\phi$.

Now we are in a position to prove the $\phi$-caloric approximation Theorem.
\begin{theorem}
  \label{thm:almost}
  Let $\sigma\in(0,1)$, $q\in [1,\infty)$ and $\theta\in(0,1)$ fixed.
  Moreover, let $\tilde{Q} =Q$ or to be more flexible let~$\tilde{Q}$
  be such that $Q \subset \tilde{Q} \subset 2Q$.  Then for
  $\epsilon>0$ there exists $\delta>0$ such that the following holds:
  if $u$ is { \sl \lq \lq almost $\phi$-caloric "} in the sense that
  for all $\xi \in C^\infty_0(Q)$,
  \begin{align}
    \label{eq:1}
    \biggabs{\dashint_Q -u \partial_t \xi + A(\nabla u)\nabla \xi \dz}
    \leq \delta \bigg( \dashint_{\tilde{Q}} \phi(\abs{\nabla u})\dz +
    \dashint_{\tilde{Q}} \phi^*(\abs{H})\dz + \phi(\norm{\nabla \xi}_\infty)
    \bigg),
  \end{align}
  then
  \begin{align*}
    & \bigg( \dashint_{I}
    \bigg(\dashint_B\Big(\frac{\abs{u-h}^{2}}{t^+-t^-}\Big)^\sigma\dx\bigg)^\frac
    q\sigma\dt\bigg)^\frac1q+\biggabs{\dashint_Q \abs{V(\nabla u) -
        V(\nabla h)}^{2\theta} \dz }^{\frac{1}{\theta}}
    \\
    &\quad \leq \epsilon \bigg( \dashint_{\tilde{Q}} \phi(\abs{\nabla u})\dz +
    \dashint_{\tilde{Q}} \phi^*(\abs{H})\dz \bigg).
  \end{align*}
\end{theorem}

\begin{proof}[Proof]
  Let $w:= u-h$ and $G=H-A(\nabla u)$.  Then 
  \begin{align*}
    \partial_t w &= \divergence G \qquad \text{on $Q$}
  \end{align*}
  and $w=0$ on $\partial_p Q$.
	We define
	 \begin{align}
\label{eq:gamma}
    \phi(\gamma) &:= \dashint_{\tilde{Q}} \phi(\abs{\nabla u})\dz + \dashint_{\tilde{Q}}
    \phi^*(\abs{H})\dz.
  \end{align}
By Lemma~\ref{lem:goodlambda} and \eqref{eq:linftyl2} we find for every $m_0\in\setN$ a $\lambda\in [\gamma,2^{m_0}\gamma]$, such that for $\alpha=\alpha(\lambda):=\frac{\lambda}{\phi'(\lambda)}$
\begin{align}
  \label{eq:goodbeta}
  \abs{\set{\mathcal{M}^{\alpha}(\nabla w\chi_Q)>\lambda}}+\abs{\set{\alpha
      \mathcal{M}^\alpha(G\chi_Q)>\lambda}} \leq
  \frac{c\phi(\gamma)}{\phi(\lambda)m_0}\abs{Q}.
\end{align}
with $c$ independent of $m_0,\gamma$ and $\lambda$.

\noindent  Now, let $\wal$ be the Lipschitz truncation of~$w$ as
  in Section~\ref{sec:truncation}, i.e.
  \begin{align*}
    \Oal := (\set{\Ma(\nabla w\chi_Q)>\lambda} \cup \set{\alpha \Ma(G\chi_Q) >
      \lambda})\text{ and } \supp(\wal)\subset \overline{\Oal\cap Q}.
  \end{align*}
%
%
%
 We use the test function $\xi = \wal \eta$, where $\eta= \max\set{\frac{t^+-t}{t^+-t^-},0}\in[0,1]$ on 
 $I=[t^-,t^+]$.  Note that in general~$\xi \notin
 C^\infty_0(Q)$. However, it follows by a simple convolution argument
 as in (4.1) of~\cite{DieStrVer12}, that the validity of~\eqref{eq:1}
 for all~$\xi \in C^\infty_0(Q)$ implies its validity under the
 assumption~$\norm{\nabla \xi}_\infty< \infty$. Thus,~$\xi$ is a valid
 test function.

Therefore, using the
Theorem~\ref{thm:orlicz}~\ref{itm:integrationbyparts} we find
\begin{align*}
  (I_1)+ (I) + (II) &:= \dashint_Q \frac{\abs{\wal}^2}{2} (-\partial_t
  \eta) \dz - \dashint_Q (w-\wal)\partial_t((\wal) \eta)\,dz
  \\
  &\qquad + \dashint_Q \langle (A(\nabla u)-A(\nabla h)), (\nabla
  \wal)\, \eta\rangle \dz
  \\
  &\leq \delta \bigg( \dashint_{\tilde{Q}} \phi(\abs{\nabla u})\dz \ds+
  \dashint_{\tilde{Q}}\phi^*(\abs{H})\dz + \phi(\norm{\eta \nabla \wal}_\infty)
  \bigg)
  \\
  &\leq \delta \bigg( \dashint_{\tilde{Q}} \phi(\abs{\nabla u})\dz + \dashint_{\tilde{Q}}
  \phi^*(\abs{H})\dz + c_{m_0}\,\phi(\gamma) \bigg)=:(III).
  \end{align*}
	using $\norm{\nabla \wal}_\infty \leq c \lambda \leq c 2^{m_0}\gamma$. As $-\partial_t \eta =\frac1{(t^+-t^-)} \geq 0$, we have that $(I_1)>0$. 
	We estimate the other terms. 
		\begin{align*}
(I)
&= - \dashint_Q
     (w-\wal)\, \partial_t\wal \eta \dz- \dashint_Q
     (w-\wal)\, \wal \partial_t\eta \dz \\
         &= -  \dashint_Q
     (w-\wal)\, \partial_t\wal \eta \dz - \frac{1}{\abs{Q}}\sum_{i}\int_{Q_i} (w-w_i^{\alpha})\rho_i\, \wal\partial_t\eta \dz \\
    &= -  \dashint_Q
     (w-\wal)\, \partial_t\wal \eta \dz - \frac{1}{\abs{Q}}\sum_{i}\int_{Q_i} (w-w_i^{\alpha})\rho_i\, \sum_{{j\in A_i}}\rho_jw_j^{\alpha}\partial_t\eta \dz \\
&\seb{= -  \dashint_Q
     (w-\wal)\, \partial_t\wal \eta \dz- \frac{1}{\abs{Q}}\sum_{\set{i:\exists j\in A_i : w^\alpha_j\neq 0}}\int_{Q_i} (w-w_i^{\alpha})\rho_i\, \sum_{{j\in A_i}}\rho_jw_j^{\alpha}\partial_t\eta \dz }\\
&=-(I_2)-(I_3).
\end{align*}
\seb{Using the fact, that $\supp(\rho_j)\subset\frac 34Q_j^{\alpha}$}, we estimate
  \begin{align*}
    (I_2) &\leq \dashint_Q \chi_{\Oal} \abs{ w-\wal}
   \seb{\abs{\partial_t \wal \eta}}\dz
    \\
    &= \frac1{\abs{Q}}\int\chi_{Q\cap\Oal} \sum_i 
\bigabs{
     \rho_i(w-w_i^{\alpha})}\bigabs{\sum_{{j\in A_i}} \partial_t\rho_i w_i^{\alpha}} \dz
    \\
    &= \frac1{\abs{Q}} \sum_i \int_{\frac34 Q_i}\chi_{Q\cap\Oal} \bigabs{
      \rho_i(w-w_i^{\alpha})}\bigabs{ \sum_{{j\in A_i}} \partial_t\rho_i
       (w_i^{\alpha}-w_j^{\alpha})} \dz
			\end{align*}
We estimate further using \ref{itm:P3}, \ref{itm:P1}, \ref{itm:whit_fat}, \ref{itm:whit_sum}, \ref{itm:whit2} Lemma~\ref{lem:diff_uj_uk}, Lemma~\ref{lem:w-w_j_lambda} the fact that $\Oal$ is symmetric around $t^+$ and \eqref{eq:goodbeta}. 
   \begin{align*}
   (I_2) 
   &\leq \frac c{\alpha\abs{Q}}\sum_i\sum_{{j\in A_i}}\int_{\frac34 Q_i}\frac{\abs{w-w_i^{\alpha}}}{r_i}
\frac{\abs{w_j^{\alpha}-w_i^{\alpha}}}{r_j}
\\
&\leq \frac c{\alpha\abs{Q}}\sum_i\sum_{{j\in A_i}}\abs{Q_i}\dashint_{\frac34Q_i}\frac{\abs{w-w_i^{\alpha}}}{r_i}\lambda\\
&\leq \frac{c\lambda^2}{\alpha}\frac{\abs{\Oal}}{\abs{Q}}\leq \frac{c\phi(\gamma)}{m_0}.
\end{align*}
To estimate $(I_3)$ 
we making use of the fact that either $w_i^\alpha=0$ or $\supp(\rho_i)\subset\frac 34Q_i^{\alpha}\subset Q$ and then $\int_{Q_i}(w-w_i^{\alpha})\rho_i \partial_t\eta\dz=0$. Since $\sum_i\rho_i= 1$ we find
\[
   (I_3) =\frac{1}{\abs{Q}}\sum_{\set{i:\exists j\in A_i : w^\alpha_j\neq 0}}\int_{Q_i} (w-w_i^{\alpha})\rho_i\, \sum_{{j\in A_i}}\rho_j(w_j^{\alpha}-w_i^{\alpha})\partial_t\eta \dz.
\]
Next, observe that there exists $j\in A_i$, such that $w^\alpha_j\neq 0$ and hence $\frac34 Q_j\subset J\times B$. This however implies that $r^2_j\alpha \leq 2(t^+-t^-)$ and consequently by \ref{itm:whit_radii} that $r^2_i\alpha \leq c(t^+-t^-)$. Using this bound together with the argument that was used to estimate $(I_2)$ implies 
\begin{align*}
   \abs{(I_3)}  &\leq \seb{\frac c{t^+-t^-}\frac{1}{\abs{Q}}\sum_{\set{i:\exists j\in A_i : w^\alpha_j\neq 0}}\int_{Q_i} \abs{w-w_i^{\alpha}}\dz\sum_{{j\in A_i}} \abs{w_i^{\alpha}-w_j^{\alpha}}}\\
   &\leq \seb{\frac{c}{\alpha \abs{Q}}\sum_{i}\int_{Q_i} \frac{\abs{w-w_i^{\alpha}}}{r_i}\dz\sum_{{j\in A_i}} \frac{\abs{w_i^{\alpha}-w_j^{\alpha}}}{r_i}}\\
 &\leq \seb{\frac{c\lambda^2}{\alpha}\frac{\abs{\Oal}}{\abs{Q}}\leq  \frac{c\phi(\gamma)}{m_0}.}
\end{align*}
Now we continue by estimating $(II)$.
  Recall that $\abs{\nabla \wal} \leq c\, \lambda$ and that
  $\wal=w=u-h$ on $Q \setminus \Oal$. This gives
  \begin{align*}
    (II) &= \dashint_Q \langle A(\nabla u)-A(\nabla h), \nabla \wal\, \eta\rangle
    \dz
    \\
    &\geq c\, \dashint_Q \chi_{Q \setminus \Oal} \abs{V(\nabla
      u)-V(\nabla h)}^2 \eta \dz - c\dashint_Q \chi_{\Oal}
    (\abs{A(\nabla u)}+\abs{A(\nabla h)}) \lambda \dz
    \\
    &=: (II_1) - (II_2).
  \end{align*}

 \noindent  Using Young's inequality with~$\tilde{\delta}$, that can be chosen independent of $m_0,\gamma,\lambda$ and \eqref{eq:goodbeta}, we find that
  \begin{align*}
    (II_2) &= \dashint_Q \chi_{\Oal} (\abs{A(\nabla u)}+\abs{A(\nabla
      h)}) \lambda \dz
    \\
&\leq  c_{\tilde{\delta}}\phi(\lambda)\frac{\abs{Q\cap\Oal}}{\abs{Q}}+{\tilde{\delta}}\dashint_Q \chi_{\Oal}\phi(\abs{\nabla u})\dz\leq c\Big(\frac{c_{\tilde{\delta}}}{m}+\tilde{\delta}\Big)\phi(\gamma).
%
  \end{align*}

\noindent  So far we have
  \begin{align*}
    (III) &= (I) + (II) \geq (I_1)-(I_2)- (I_3) + (II_1)-(II_2)
  \end{align*}
  which implies by that
  \begin{align}
\label{eq:II1}
\begin{aligned}
    (II_1) +(I_1)&\leq (II_2) + \abs{(I_2)+(I_3)} + (III)
    \\
    &\leq \Big(c_{m_0}\delta +3\tilde{\delta}+\frac{c_{\tilde{\delta}}}{m_0}\Big)\phi(\gamma)
\end{aligned}
\end{align}
%
  Observe, that for $\beta\in (0,1)$ we find
  \[
\seb{\bigg(\dashint_{t^-}^{t^+}\eta^{-\beta}\dt\bigg)^\frac1\beta
=\bigg(-\dashint_{t^--t^+}^{0}\frac{(t^+-t^-)^\beta}{s^\beta}\ds\bigg)^\frac1\beta
=(t^+-t^-)^\frac{\beta-1}\beta\bigg(\int_0^{t^+-t^-}{s^{-\beta}}\ds\bigg)^\frac1\beta=(1-\beta)^\frac{-1}{\beta}.
}
  \]
  Now we fix $\theta\in(0,\frac12)$, such that $\beta=\frac{1}{1-\theta}\in (0,1)$. For $\theta$ closer to $1$, we will later use an interpolation with \eqref{eq:linftyl2}. For this fixed $\theta\in(0,\frac12)$ we get by the above that
  \begin{align*}
    (IV) &:= \bigg(\dashint_Q \abs{V(\nabla u)-V(\nabla h)}^{2\theta} \dz
    \bigg)^{\frac 1 \theta}\\
&=\bigg(\dashint_Q \chi_{\Oal}\abs{V(\nabla u)-V(\nabla h)}^{2\theta} \dz
    +\dashint_Q \chi_{(\Oal)^c}\abs{V(\nabla u)-V(\nabla h)}^{2\theta} \dz
    \bigg)^{\frac 1 \theta}.\\
&\leq  c\, \dashint_Q \abs{V(\nabla u)-V(\nabla h)}^2\dz
    \bigg(\frac{\abs{Q\cap\Oal}}{\abs{Q}} \bigg)^\frac{1-\theta}\theta\\
&\quad+c\dashint_Q \chi_{(\Oal)^c}\abs{V(\nabla u)-V(\nabla h)}^{2}\eta \dz
\seb{\bigg(\dashint_Q\chi_{(\Oal)^c}\eta^\frac{-\theta}{1-\theta}\dz \bigg)^\frac{1-\theta}{\theta}}
\\
&\leq  c\, \dashint_Q \abs{V(\nabla u)-V(\nabla h)}^2\dz
    \bigg(\frac{\abs{Q\cap\Oal}}{\abs{Q}} \bigg)^\frac{1-\theta}{\theta}+\seb{c\Big(\frac{1-\theta}{1-2\theta}\Big)^\frac{1-\theta}{\theta}}(II_1)
    \\
    &=: c(V)+c(II_1).
  \end{align*}
  Now, by Lemma~\ref{lem:goodlambda} we get
  \begin{align}
\label{eq:V1}
\begin{aligned}
    (V) &\leq  c\, \dashint_Q \phi(\abs{\nabla u}) +
    \phi(\abs{\nabla h}) \dz \bigg(\frac{\abs{Q\cap\Oal}}{\abs{Q}}
    \bigg)^\frac{1-\theta}\theta.
    \\
    &\leq  \Big(\frac{c\phi(\gamma)}{\phi(2^{m_0}\gamma)m_0}\Big)^\frac{1-\theta}\theta\, \dashint_Q \phi(\abs{\nabla u}) + \phi^*(\abs{H})\\
  &\leq \frac{c\phi(\gamma)}{2^\frac{m_0(1-\theta)}{\theta}}
\end{aligned}  
\end{align}
For the estimate from below for $(I_1)$ we estimate similarly that 
\begin{align*}
  (VI)&:=\dashint_{I}\Big(\dashint_B\frac{\abs{w}}{\sqrt{t^+-t^-}}
  \dx\bigg)^2\dt 
  \\
  &\leq \dashint_Q\chi_{({\Oal})^c}
  \frac{\abs{\wal}^{2}}{t^+-t^-}\dz
+ \dashint_{I}\frac{\abs{(\set{t}\times{B})\cap
      \Oal}}{\abs{B}}
  \dashint_{\set{t}\times B}\frac{\abs{w}^{2}}{t^+-t^-}\dx\dt
  \\
  &\leq c(I_1)+\frac{\abs{Q \cap
      \Oal}}{\abs{Q}}
  \sup_{I}\dashint_B\frac{\abs{w}^{2}}{t^+-t^-}\dx.
\end{align*}
Now we use \eqref{eq:linftyl2}, \eqref{eq:gamma} and Lemma~\ref{lem:goodlambda} to find that
\begin{align*}
  (VI)
  &\leq c(I_1)+c\frac{\phi(\gamma)}{m_0\phi(2^{m_0}\gamma)}\phi(\gamma)\leq \frac{\phi(\gamma)}{2^{m_0}}
\end{align*}
This implies together with Lemma~\ref{lem:goodlambda}, \eqref{eq:II1}
and \eqref{eq:V1}

\[
 (VI)+(IV)\leq \Big(c_{m_0}\delta +3\tilde{\delta}+\frac{c_{\tilde{\delta}}}{m_0}+\frac{c}{2^\frac{m_0(1-\theta)}{\theta}}+\frac{c}{2^{m_0}}\Big)\phi(\gamma).
\]
Let us fix the auxiliary constant $\tilde{\epsilon}\in(0,1)$. It shall be fixed at the very end of the proof. 
In the following order we choose $\tilde{\delta},m_0$ and $\delta$.
We choose $\tilde{\delta}=\frac{\tilde{\epsilon}}{5}$. Then we choose $m_0$ large enough, such that $\frac{c_{\tilde{\delta}}}{m_0}+\frac{c}{2^\frac{m_0(1-\theta)}{\theta}}+\frac{c}{2^{m_0}}\leq \frac{\tilde{\epsilon}}{5}$. Finally we fix $\delta$ small enough such that $c_{m_0}\delta\leq \frac{\tilde{\epsilon}}{5}$. These choices imply the following estimate for a fixed $\theta\in (0,\frac12)$ (for example $\theta=\frac14$)
\[
 \dashint_{I}\Big(\dashint_B\frac{\abs{w}}{\sqrt{t^+-t^-}}
  \dx\bigg)^2\dt + \bigg(\dashint_Q \abs{V(\nabla u)-V(\nabla h)}^{2\theta} \dz
    \bigg)^{\frac 1 \theta} \leq \epsilon \phi(\gamma).
\]
Finally, by interpolation between \seb{the estimate above and estimate~\eqref{eq:linftyl2}, we find the result. For the sake of completion we include the interpolation between $L^2(L^1)$ and the $L^\infty(L^2)$ estimate. The interpolation between $L^2$ and $L^{2\theta}$ for the gradient terms is similar but more straight forward, such that we omit the details.} Let us fix $f=\frac{\abs{w}}{\sqrt{t^+-t^-}}$. Then we find for $b\in (2,\infty)$ and $a\in (1,2)$ by H\"older, Jensen's inequality, \eqref{eq:linftyl2} and \eqref{eq:gamma} that
\begin{align*}
&\bigg(\dashint_{I}\bigg(\dashint_{B}\Big(\frac{\abs{w}}{\sqrt{t^+-t^-}}\Big)^a\dx\bigg)^\frac{b}{a}\dt\bigg)^\frac{2}{b}
=
\bigg(\dashint_{I}\bigg(\dashint_{B}\abs{f}^{2-a}\abs{f}^{2(a-1)}\dx\bigg)^\frac{q}{a}\dt\bigg)^\frac{2}{b}
\\
&\leq
\bigg(\dashint_{I}\bigg(\dashint_{B}\abs{f}\dx\bigg)^\frac{q(2-a)}{a}\bigg(\dashint_{B}\abs{f}^2\dx\bigg)^\frac{q(a-1)}a\dt\bigg)^\frac{2}{b}
\\
&\leq\bigg(\dashint_{I}\bigg(\dashint_{B}\abs{f}^2\dx\bigg)^{\frac{q(a-1)}a+\frac{q(2-a)}{2a}-1}\bigg(\dashint_{B}\abs{f}\dx\bigg)^2\dt\bigg)^\frac{2}{b}
\\
\quad&\leq \sup_{I}\bigg(\dashint_{B}\abs{f}^2\dx\bigg)^{\frac{b-2}2\frac{2}{b}}\bigg(\dashint_{I}\bigg(\dashint_{B}\abs{f}\dx\bigg)^2\dt\bigg)^\frac{2}{b}\leq c_0^\frac{q-2}{q}\tilde{\epsilon}^\frac{2}{q}\phi(\gamma).
\end{align*}
Choosing $a=2\sigma$ and $b=2q$ the proof is completed by an appropriate choice of $\tilde{\epsilon}$.
\end{proof}

\bibliographystyle{amsalpha}

\end{document}